\documentclass[reqno]{amsart}
    \usepackage[margin=0.7in]{geometry}
    \usepackage[parfill]{parskip}
    \usepackage[utf8]{inputenc}
    \usepackage{pgfplots}
\usepackage{graphicx}
\usepackage{subcaption}
\usepackage{natbib}
\usepackage{amsfonts}
\usepackage{amsmath}
\usepackage{amssymb}
\usepackage{amsthm}
\usepackage{setspace}
\usepackage{xcolor}
\usepackage{tikz}
\usepackage{mathrsfs} 
\usepgflibrary{patterns}

\makeatletter
\newcommand{\mylabel}[2]{#2\def\@currentlabel{#2}\label{#1}}
\makeatother

\usepackage[draft]{todonotes}   
\usepackage{hyperref}
\newtheorem{lemma}{Lemma}

\newtheorem{theorem}{Theorem}
\newtheorem{remark}{Remark}

\usetikzlibrary{plotmarks}

\numberwithin{equation}{section}

\title{SI-method for solving stiff nonlinear boundary value problems}
\author{Volodymyr Makarov }
\author{Denys Dragunov }
\thanks{Address : Insitute of Mathematics, National Academy of Sciences of Ukraine, 01024 Ukraine, Kiev-4,
3, Tereschenkivska st.}
\thanks{E-mail: makarov@imath.kiev.ua (Volodymyr Makarov); dragunovdenis@imath.kiev.ua (Denys Dragunov, corresponding author)}


\begin{document}
\maketitle

\begin{abstract}
The paper contains a thorough theoretical analysis of the SI-method, which was firstly introduced in \cite{Makarov_Dragunov_2019} and proved to be remarkably stable and efficient when applied to some instances of stiff boundary value problems (like the Troesch's problem). By suggesting a more general view on the SI-method's idea and framework, we managed to obtain sufficient conditions for the method to be applicable to a certain class of two-point boundary value problems. The corresponding error estimates are provided. Special attention is devoted to the exploration of the method's capabilities via a set of numerical examples. The implementation details of the method are discussed in fair depth. An open-source C++ implementation of the SI-method is freely available at the public repository \url{https://github.com/imathsoft/MathSoftDevelopment}.
\end{abstract}

\Large
\section{Introduction}
The aim of the present paper is to provide a thorough theoretical justification of the SI-method proposed in \cite{Makarov_Dragunov_2019}. In what follows, we give a slightly broader view on the SI-method, as compared to that from \cite{Makarov_Dragunov_2019}, and obtain sufficient conditions ensuring the method's applicability to a certain class of two-point boundary value problems.

The paper is focused on a boundary value problem (BVP) of the form
\begin{equation}\label{Intro_Equation}
  u^{\prime\prime}(x) = \mathcal{N}(u(x), x),
\end{equation}
\begin{equation}\label{Intro_boundary_conditions}
  u(a) = 0, \; u(b) = u_{b}, \; a<b,\;  0 < u_{b} \in \mathbb{R},
\end{equation}
which finds a number of applications in physics and is the object of a great many studies in numerical analysis, see, for example, \cite{Sweidan_Mohyeedden_Chen_Xiaojun_Zheng_Xiaoming_2P_BVP_2020}, \cite{Bhal_Santosh_Kumar_Danumjaya_Fairweather_2P_BVP_2020}, \cite{Mohanty_Manchanda_Geetan_Khan_Arshad_Khurana_Gunjan_2P_BVP_2020}, \cite{Ma_Guanglong_Stynes_Martin_2P_BVP_2020}, \cite{Ghorbani_Asghar_Passandideh_Hadi_2P_BVP_2020}, \cite{Kiguradze_2P_BVP_2019}, \cite{Justine_Hynichearry_Chew_Jackel_Vui_Ling_Sulaiman_Jumat_2P_BVP_2017} and the references therein.
Additionally, we assume
\begin{equation}\label{Intro_Nonlin_cond}
  \mathcal{N}(u, x) \equiv N(u, x)u,\; N(u, x) \in C^{1}(\mathbb{R}\times [a,b]),\; \mathcal{N}^{\prime}_{u}(u, x) \geq 0,\; \forall x\in [a,b],\; \forall u \in \mathbb{R},
\end{equation}
which guarantees existence and uniqueness of the solution to BVP \eqref{Intro_Equation}, \eqref{Intro_boundary_conditions} (see, \cite[p. 331, Theorem 7.26]{Kelley_Peterson_2010}).

As it was pointed out in \cite{Lee_June_Yub_and_Greengard_Leslie_1997_SIAM}, problems of type \eqref{Intro_Equation}, \eqref{Intro_boundary_conditions} can exhibit many different phenomena, including boundary layers, dense oscillations, and complicated or ill-conditioned internal transition regions. Any of the mentioned "complications" results in the solution process being rather expensive and unstable, which, in turn, characterizes the corresponding problem as being {\it stiff.} There were several different attempts to define the {\it stiffness} as such and those are fairly well summarized in \cite{Fifty_years_of_stiffness}.

Being not uncommon in the physics realm (see \cite{Hairer_Wanner_Stiff}), stiff BVPs has received a great deal of attention from the side of computational mathematics for the last (at least) five decades. To some extent, the essential part of almost all the numerical methods for solving stiff BVPs consists in the "construction of a mesh on which all features of the solution are locally smooth" (see \cite{Lee_June_Yub_and_Greengard_Leslie_1997_SIAM}). The latter can be achieved, for example, (i) by introducing a {\it monitor function} and building the mesh in such a way that the function is "equidistributed" on it (see \cite{Wright_Cash_Moore_1994}, \cite{Lee_June_Yub_and_Greengard_Leslie_1997_SIAM}); (ii) by applying a smooth transformation of the independent variable such that in this transformed coordinate, a number of derivatives of the solution are bounded (see \cite{Kreiss_Nichols_Brown}); (iii) by introducing a smooth transformation of the unknown solution so that the transformed problem can be solved on a more-less uniform mesh (see \cite{Chang20103043}, \cite{Chang20103303}, \cite{Gen_sol_of_TP}). The SI-method is not an exception and, in a way, its crucial part is also concerned with building a "proper" mesh, though, it does this in a rather specific manner.

In the present paper we are primarily interested in the cases when problem \eqref{Intro_Equation}, \eqref{Intro_boundary_conditions} is stiff, in the particular sense that its solution $u(x)$ possesses narrow intervals of rapid variation, known as the {\it boundary layers}. In \cite{Makarov_Dragunov_2019} the general idea of the SI ("straight-inverse") method was suggested for tackling problems of this kind. The approach is based on a simple observation that inside the boundary layers, where $|u^{\prime}(x)| \gg 1,$ the inverse function $x(\cdot) = u^{-1}(\cdot)$ (which, obviously, exists) is close to a constant. The latter means that switching to the problem with respect to the inverse function $x(u),$ whenever the straight function changes rapidly, is beneficial from the computational point of view.
Getting back to the "mesh construction" discussion, we can say that by switching between the problems for "straight" $u(x)$ and "inverse" $x(u)$ unknown functions we can keep our meshes almost uniform (each in its own dimension: $"x"$ or $"u"$).

Despite all the generality, simplicity and efficiency of the SI-method demonstrated in \cite{Makarov_Dragunov_2019}, the  latter work has not provided the necessary theoretical justification of the method in order to answer questions about its range of applicability and approximation properties. Here we aim to start filling this theoretical gap, admitting, however, that to cover the subject in depth definitely requires more than one publication.

The paper is organized as follows. In Section \ref{section_straight_inverse_hybrid} we explain the essence of the SI-method's idea through the concepts of "straight", "inverse" and "hybrid" problems. We show that the "hybrid" problem is more accessible from the computational point of view and has a unique solution which partly coincide with that of the original ("straight") problem. In Section \ref{section_SI_method_and_its_numerical_proeprties} we introduce a numerical scheme for solving the "hybrid" problem and investigate its properties. Section \ref{section_error_analysis} is devoted to the error analysis of the mentioned numerical scheme. Using the results from \cite{Vidossich_Giovanni_2001_differentiability} we prove Theorem \ref{Main_theorem_abour_apptoximation_prorerties_of_the_SI_method} about approximation properties of the SI-method applied to BVP \eqref{Intro_Equation}, \eqref{Intro_boundary_conditions}, which, in effect, specifies the statement of Proposition 2 formulated in \cite{Makarov_Dragunov_2019} without a proof. Implementation aspects of the SI-method are discussed in Section \ref{section_implementation_aspects} and the numerical examples are presented in Section \ref{section_numerical_examples}. Section \ref{Section_conclusions} contains our conclusions.

\section{The "straight", "inverse" and "hybrid" problems.}\label{section_straight_inverse_hybrid}

In what follows we assume that together with condition \eqref{Intro_Nonlin_cond}, which provides the existence of the solution, the nonlinearity in the right hand side of equation \eqref{Intro_Equation} satisfies the inequality
   \begin{equation}\label{lema_1_positiveness_of_nonlinearity}
     N(u, x)\geq 0,\; \forall u\in [0, +\infty),\; x\in[a,b],
   \end{equation}
which makes the solution's behaviour more predictable, as it is stated by the lemma below.

\begin{lemma}\label{lemma_monotonicity_of_u}
   Let conditions \eqref{Intro_Nonlin_cond} and \eqref{lema_1_positiveness_of_nonlinearity} hold true.
   Then the solution $u(x)$ to BVP \eqref{Intro_Equation}, \eqref{Intro_boundary_conditions} is  monotonically increasing and convex on $[a, b].$
\end{lemma}
\begin{proof}
First of all, let us point out that
\begin{equation}\label{lemma_proof_derivative_is_not_0_at_point_a}
  u^{\prime}(a) \neq 0.
\end{equation}
Otherwise, according to the the Pickard-Lindelof Theorem (see, for example, \cite[p.350]{Kelley_Peterson_2010}), whose conditions are fulfilled, $u(x)$ must totally coincide with $0$ on $[a,b],$ which contradicts the condition $u(b) = u_{b} > 0$ (see \eqref{Intro_boundary_conditions}).

Second, let us prove that $u(x) > 0,$ $\forall x\in (a,b].$ Assume that the latter is not true and there exists at least one point $x_{1}\in(a,b)$ such that $u(x_{1}) \leq 0.$ This immediately implies the existence of point $b_{1} \in [x_{1},b),$  such that
\begin{equation}\label{lemma_proof_assumption}
  u(a) = u(b_{1}) = 0,\; x_{1}\in (a, b_{1}].
\end{equation}

Obviously, function $w(x)\equiv 0$ satisfies equation \eqref{Intro_Equation} which, in conjunction with condition \eqref{Intro_Nonlin_cond}, allows us to apply the result of Theorem 21 from \cite[p. 48]{Protter_Weinberger_Max_principle} (the maximum principle) and prove that neither $u(x)$ nor $-u(x)$ can achieve positive maximum on $[a, b_{1}]$ and, hence, $u(x) = 0,$ $\forall x\in [a, b_{1}].$ The latter means that $u^{\prime}(0) = 0,$ which contradicts to \eqref{lemma_proof_derivative_is_not_0_at_point_a}!?

The fact that $u(x)$ is positive on $(a,b]$ together with condition \eqref{lema_1_positiveness_of_nonlinearity} means that
\begin{equation}\label{lemma_proof_second_deriv_is_nonnegative}
  u^{\prime\prime}(x) \geq 0, \; \forall x\in [a,b].
\end{equation}

On the other hand, in the light of \eqref{lemma_proof_derivative_is_not_0_at_point_a}, the positiveness of $u(x)$ on $(a,b]$ immediately yields us
\begin{equation}\label{lemma_proof_positiveness_of_derivative_at_point_a}
  u^{\prime}(a) > 0.
\end{equation}

Combining \eqref{lemma_proof_second_deriv_is_nonnegative} and \eqref{lemma_proof_positiveness_of_derivative_at_point_a} we get the statement of the Lemma.
\end{proof}

From Lemma \ref{lemma_monotonicity_of_u} it follows that, under conditions \eqref{Intro_Nonlin_cond} and \eqref{lema_1_positiveness_of_nonlinearity}, the solution $u(x)$ of BVP \eqref{Intro_Equation}, \eqref{Intro_boundary_conditions} can have at most one boundary layer, which (if exists) must be near the point $x=b.$
 The lemma also guarantees that the solution is invertible on $[a, b].$ It is not difficult to verify that the "inverse" function $u^{-1}(\cdot) = x(\cdot)$ must be a solution to BVP
\begin{equation}\label{Intro_Equation_Inverse}
   x^{\prime\prime}(u) = -\mathcal{N}(u, x(u))\left(x^{\prime}(u)\right)^{3},\; u\in [0, u_{b}],
\end{equation}
\begin{equation}\label{Intro_boundary_conditions_Inverse}
  x(0) = a, \; x(u_{b}) = b.
\end{equation}

A few statements below give us some insight on the properties of the "inverse" problem \eqref{Intro_Equation_Inverse}, \eqref{Intro_boundary_conditions_Inverse}.

\begin{lemma}\label{lemma_represent_of_solution_inv_problem}
Let $\mathcal{N}(u,x)\in C^{1}([0, u_{b}]\times \mathbb{R})$ and function $x_{\ast}(u)\in C^{2}([0, u_{b}])$ be a solution to equation \eqref{Intro_Equation_Inverse}, then
\begin{equation}\label{Bernoully_representation_of_the_inverse_solution}
  x_{\ast}(u) = x_{\ast}(0) + \int\limits_{0}^{u}\frac{x_{\ast}^{\prime}(0)d\eta}{\sqrt{1+2(x_{\ast}^{\prime}(0))^{2}\int\limits_{0}^{\eta}\mathcal{N}(\xi, x_{\ast}(\xi))d\xi}}
\end{equation}
\end{lemma}
\begin{proof}
 Let us consider an auxiliary initial value problem
 \begin{equation}\label{lemma_represent_aux_problem}
 x^{\prime\prime}(u) = - \mathcal{N}(u, x_{\ast}(u))(x^{\prime})^{3}, \; x(0) = x_{\ast}(0),\; x^{\prime}(0) = x^{\prime}_{\ast}(0).
 \end{equation}
 From the assumptions of the lemma it follows that function $x_{\ast}(u)$ is a solution to problem \eqref{lemma_represent_aux_problem}.

 On the other hand, it is easy to see that the function in the right hand side of equality \eqref{Bernoully_representation_of_the_inverse_solution} is two times continuously differentiable in some vicinity of point $u=0$ and also satisfies problem \eqref{lemma_represent_aux_problem} (see \cite[\textit{0.1.2-6. Bernoulli equation}]{zaitsev2002handbook}).
 The {\it Picard-Lindelof Theorem} (see, for example, \cite[p. 350]{Kelley_Peterson_2010}), whose conditions are fulfilled for the case of problem \eqref{lemma_represent_aux_problem}, states that $x_{\ast}(u)$ is the unique (!) solution to IVP \eqref{lemma_represent_aux_problem}. The latter immediately yields us identity \eqref{Bernoully_representation_of_the_inverse_solution} and concludes the proof.
\end{proof}

\begin{lemma}\label{lemma_monotone_invrse_solution}
Let conditions of Lemma \ref{lemma_represent_of_solution_inv_problem} hold true and function $x_{\ast}(u)\in C^{2}([0, u_{b}])$ be a solution to BVP \eqref{Intro_Equation_Inverse}, \eqref{Intro_boundary_conditions_Inverse} with $a \neq b$, then
$$x_{\ast}^{\prime}(u) \neq 0,\;  \forall u\in[0, u_{b}].$$
\end{lemma}
\begin{proof}
From the representation \eqref{Bernoully_representation_of_the_inverse_solution} it follows that if $x^{\prime}_{\ast}(u) = 0$ at some point $u\in [0, u_{b}],$ then the same is true for every point of  interval $[0, u_{b}]$ and $x_{\ast}(u) = \textit{const}.$ The latter is impossible since $x_{\ast}(0) = a \neq b = x_{\ast}(u_{b}).$ The contradiction completes the proof.
\end{proof}

We see that under conditions \eqref{Intro_Nonlin_cond} and \eqref{lema_1_positiveness_of_nonlinearity} the solution to BVP \eqref{Intro_Equation}, \eqref{Intro_boundary_conditions} is unique, invertible and the inverse function is a solution to BVP \eqref{Intro_Equation_Inverse}, \eqref{Intro_boundary_conditions_Inverse}. The theorem below states that the opposite is also true.

\begin{theorem}\label{theorem_uniqueness_of_solution_to_inverse_problem}
  Let conditions \eqref{Intro_Nonlin_cond} and \eqref{lema_1_positiveness_of_nonlinearity} hold true, then BVP \eqref{Intro_Equation_Inverse}, \eqref{Intro_boundary_conditions_Inverse} has a unique monotone solution whose inverse is the solution to BVP \eqref{Intro_Equation}, \eqref{Intro_boundary_conditions}.
\end{theorem}
\begin{proof}
The existence of a solution to BVP \eqref{Intro_Equation_Inverse}, \eqref{Intro_boundary_conditions_Inverse} follows from the existence and monotonicity of the solution $u(x)$ to BVP \eqref{Intro_Equation}, \eqref{Intro_boundary_conditions}.
The {\it Whitney's extension theorem} (see \cite[Theorem I]{Whitney_extensions}) guarantees that function $\mathcal{N}(u,x)$ can be extended to a function from $C^{1}(\mathbb{R}^{2}).$ For each such an extension, the conditions of Lemma \ref{lemma_monotone_invrse_solution} hold true, which means that any solution $x(u)\in C^{2}([0, u_{b}])$ to BVP \eqref{Intro_Equation_Inverse}, \eqref{Intro_boundary_conditions_Inverse} with an extended $\mathcal{N}(u,x)$ is invertible and (as it can be easily verified) the two times continuously differentiable inverse function must satisfy problem \eqref{Intro_Equation}, \eqref{Intro_boundary_conditions}, coinciding with its unique solution $u(x)$. The latter yields us the uniqueness of $x(u)$.
\end{proof}

We see that there is a strong and unambiguous connection between the "straight" \eqref{Intro_Equation}, \eqref{Intro_boundary_conditions} and "inverse" \eqref{Intro_Equation_Inverse}, \eqref{Intro_boundary_conditions_Inverse}  problems. As it was mentioned above, if solution $u(x)$ has a boundary layer near the point $x=b,$ then solution $x(u)$ is close to a constant near the point $u=u_{b}.$ To utilize this remarkable property we need to consider a one-parameter family of "hybrid" problems defined as follows:
for the given value of parameter $c\in (a, b),$ find a pair of two times continuously differentiable functions $\mathfrak{u}(x)$ and $\mathfrak{x}(u)$ such that
\begin{equation}\label{c_family_problem_straight}
  \mathfrak{u}^{\prime\prime}(x) = \mathcal{N}(\mathfrak{u}(x), x), \; x\in [a, c],\; \mathfrak{u}(a) = 0,\; \mathfrak{u}(c) \neq u_{b},
\end{equation}
\begin{equation}\label{c_family_problem_inverse}
  \mathfrak{x}^{\prime\prime}(x) = -\mathcal{N}(u, \mathfrak{x}(u))\left(\mathfrak{x}^{\prime}(u)\right)^{3}, \; u\in [\mathfrak{u}(c), u_{b}],\; \mathfrak{x}(u_{b}) = b,
\end{equation}
\begin{equation}\label{c_family_problem_match_conditions}
\mathfrak{x}(\mathfrak{u}(c)) = c,\; \mathfrak{x}^{\prime}(\mathfrak{u}(c)) = \frac{1}{\mathfrak{u}^{\prime}(c)}.
\end{equation}

The following theorem reveals how the solution of "hybrid" problem \eqref{c_family_problem_straight}, \eqref{c_family_problem_inverse}, \eqref{c_family_problem_match_conditions} relates to the solutions of the "straight" and "inverse" problems.

\begin{theorem}\label{theorem_meaning_of_c_family_problem}
Let conditions \eqref{Intro_Nonlin_cond} and \eqref{lema_1_positiveness_of_nonlinearity} hold true, then for any $c\in(a,b)$ there exists a unique pair of functions $\mathfrak{u}(x)$ and $\mathfrak{x}(u)$ satisfying conditions \eqref{c_family_problem_straight}, \eqref{c_family_problem_inverse}, \eqref{c_family_problem_match_conditions}. Furthermore, the following identities hold true
$$\mathfrak{u}(x) = u(x), \; \forall x\in [a, c],$$
$$\mathfrak{x}(u) = x(u),\; \forall u\in [\mathfrak{u}(c), u_{b}],$$
where $u(x)$ and $x(u)$ are the solutions to BVPs \eqref{Intro_Equation}, \eqref{Intro_boundary_conditions} and  \eqref{Intro_Equation_Inverse}, \eqref{Intro_boundary_conditions_Inverse} respectively.
\end{theorem}
\begin{proof}
For any given $c\in (a,b)$ we can easily construct a pair of two times continuously differentiable functions $\mathfrak{u}(x)$ and $\mathfrak{x}(u)$ which is a solution to "hybrid" problem \eqref{c_family_problem_straight}, \eqref{c_family_problem_inverse}, \eqref{c_family_problem_match_conditions}. Indeed, since, according to Lemma \ref{lemma_monotonicity_of_u}, function $u(x)$ is monotone, the pair defined like this
$$\mathfrak{u}(x) = u(x), \; \forall x\in [a, c],$$
$$\mathfrak{x}(u) = u^{-1}(u),\; \forall u\in [\mathfrak{u}(c), u_{b}]$$
fulfills all the requirements. The existence is proved.

Now, let $\mathcal{N}(u,x)$ be an extension of the original right hand side function, belonging to $C^{1}(\mathbb{R}^{2})$ (it exists according to the {\it Whitney's extension theorem}, \cite[Theorem I]{Whitney_extensions}).
For any pair of functions $\mathfrak{u}(x),$ $\mathfrak{x}(u)$ satisfying conditions \eqref{c_family_problem_straight}, \eqref{c_family_problem_inverse}, \eqref{c_family_problem_match_conditions} we can consider an auxiliary function
$$u_{\ast}(x) = \left\{
                  \begin{array}{cc}
                    \mathfrak{u}(x), & \forall x\in [a, c], \\
                    \mathfrak{x}^{-1}(u), & \forall x\in (c, b], \\
                  \end{array}
                \right.
$$
which, according to Lemma \ref{lemma_monotone_invrse_solution}, whose conditions are obviously fulfilled with $x_{\ast}(u) = \mathfrak{x}(u)\in C^{2}([\mathfrak{u}(c), u_{b}])$, is well defined. From conditions \eqref{c_family_problem_straight}, \eqref{c_family_problem_inverse}, \eqref{c_family_problem_match_conditions} it follows that function $u_{\ast}(x)$  belongs to $C^{2}([a, b])$ and satisfies BVP \eqref{Intro_Equation}, \eqref{Intro_boundary_conditions}, which has a unique solution. This yields the uniqueness of the pair $\mathfrak{u}(x), \mathfrak{x}(u)$ and thus completes the proof.
\end{proof}

Let $[b-\varepsilon , b]$ be a narrow (i.e., $(b-a)/\varepsilon \gg 1$) interval of rapid variation for the solution $u(x),$ where $u^{\prime}(x) > 1$ and $u^{\prime}(b) \gg 1,$ also known as the {\it boundary layer}. Then, for $c = b-\varepsilon,$ problem \eqref{c_family_problem_straight}, \eqref{c_family_problem_inverse}, \eqref{c_family_problem_match_conditions} is non-stiff (or considerably less stiff, as compared to the original problem \eqref{Intro_Equation}, \eqref{Intro_boundary_conditions}). Indeed, as it was pointed out above, $[b-\varepsilon , b]$ is the only boundary layer of the solution $u(x),$ provided that conditions \eqref{Intro_Nonlin_cond} and \eqref{lema_1_positiveness_of_nonlinearity} are satisfied. The latter allows us to conclude that the variation of $\mathfrak{u}(x) = u(x)$ on $[a, c]$ is rather moderate. The same is true with respect to function $\mathfrak{x}(u) = x(u),$ since, apparently, $0 < x^{\prime}(u) < 1,$ $\forall u\in [u(c), u_{b}].$

The above property of the "hybrid" problem is the key to the SI-method, which, instead of approximating the solution of the original (potentially stiff) BVP \eqref{Intro_Equation}, \eqref{Intro_boundary_conditions}, solves a non-stiff (less stiff) problem \eqref{c_family_problem_straight}, \eqref{c_family_problem_inverse}, \eqref{c_family_problem_match_conditions}. In such a way, by solving a simpler, from the computational point of view, problem we still get the solution of a more complex problem partially approximated (Theorem \ref{theorem_meaning_of_c_family_problem}). Granted, the SI-method does not allow us to approximate solution $u(x)$ on $[c, b].$ The latter, however, is a fundamental problem: to get an efficient approximation of a function on an interval where its derivatives can take arbitrary big absolute values.

\section{SI-method: numerical aspect}\label{section_SI_method_and_its_numerical_proeprties}
In this section we describe and justify a numerical scheme (one out of many possible) for solving the "hybrid" problem introduced above.

Let $c$ be some arbitrary fixed point from $(a, b).$ In order to approximate the solution of "hybrid" problem \eqref{c_family_problem_straight}, \eqref{c_family_problem_inverse}, \eqref{c_family_problem_match_conditions} we suggest to divide the intervals $[a, c]$ and $[0, u_{b}]$ into subintervals
\begin{equation}\label{discretization_x}
  \delta_{x} = \left\{ [x_{i-1}, x_{i}],\; i\in \overline{1,\ N_{1}} \right\},
\end{equation}
$$ x_{0} = a,\; x_{N_{1}} = c,\; x_{i-1} < x_{i}\; \forall i\in \overline{ 1, N_{1}}$$
and
\begin{equation}\label{discretization_u}
  \delta_{u} = \left\{ [u_{i-1}, u_{i}], \; i\in\overline{1, N_{2}}\right\}
\end{equation}
$$u_{0} = 0,\; u_{N_{2}} = u_{b},\; u_{i-1} < u_{i},\; \forall i\in\overline{1, N_{2}}$$
respectively.

Consider a pair of functions $\tilde{u}(x)$ and $\tilde{x}(u)$ satisfying the following conditions:
\begin{enumerate}
  \item \label{condition_one} function $\tilde{u}(x)$ is a solution to the equation
\begin{equation}\label{discretized_equation_straight}
  \tilde{u}^{\prime\prime}(x) = \alpha(\mathbb{P}_{x}[\tilde{u}^{\prime}(x)], \mathbb{P}_{x}[\tilde{u}(x)], x)\tilde{u}(x),\; x\in [a, c],\; \tilde{u}(x) \in C^{1}([a, c]),
\end{equation}
where
$$\mathbb{P}_{x}[f(x)] = f(x_{i}),\; x\in [x_{i}, x_{i+1}], \; \forall i\in \overline{0, N_{1}-1},$$
\begin{equation}\label{discretized_equation_straight_term_definition}
  \alpha(u^{\prime}, u, x) = \alpha_{i}(u^{\prime}, u, x)
\end{equation}
$$= \left(N^{\prime}_{u}(u, x_{i})u^{\prime} + N^{\prime}_{x}(u, x_{i})\right) (x-x_{i}) + N(u, x_{i}), \; x\in [x_{i}, x_{i+1}], \; \forall i \in \overline{0, N_{1}-1}$$
and satisfy the inequality
\begin{equation}\label{c_inequality_for_straight_discretization}
  0 < \tilde{u}(c) < u_{b};
\end{equation}
  \item \label{condition_two} function $\tilde{x}(u)$ is a solution to the equation
\begin{equation}\label{discretized_equation_inverse}
  \tilde{x}^{\prime\prime}(u) = \beta(\mathbb{P}_{u}[\tilde{x}^{\prime}(u)], \mathbb{P}_{u}[\tilde{x}(u)], u)\left(\tilde{x}^{\prime}(u)\right)^{3},\; u\in [\tilde{u}(c), b],\; \tilde{x}(u) \in C^{1}([\tilde{u}(c), b]),
\end{equation}
where
$$\mathbb{P}_{u}[f(u)] = f(\bar{u}_{i}),\; u\in [\bar{u}_{i}, \bar{u}_{i+1}],\; \forall i\in\overline{0, N_{2}-1}$$
\begin{equation}\label{discretized_equation_inverse_term_definition}
  \beta(x^{\prime}, x, u) = \beta_{i}(x^{\prime}, x, u)
\end{equation}
$$= -\left(\mathcal{N}^{\prime}_{u}(\bar{u}_{i}, x) + \mathcal{N}^{\prime}_{x}(\bar{u}_{i}, x)x^{\prime}\right)(u-\bar{u}_{i}) - \mathcal{N}(\bar{u}_{i}, x),\; u\in [\bar{u}_{i}, \bar{u}_{i+1}],\; \forall i\in\overline{0, N_{2}-1}$$
\begin{equation}\label{definition_of_u_bar}
  \bar{u}_{i} = \bar{u}_{i}(\tilde{u}(c)) = \left\{
       \begin{array}{cc}
         u_{i}, & u_{i}> \tilde{u}(c), \\
         \tilde{u}(c), & u_{i}\leq \tilde{u}(c), \\
       \end{array}
     \right.\; \forall i\in\overline{0, N_{2}-1}.\footnote{As one can notice, $\bar{u}_{i} = \bar{u}_{j} = \tilde{u}(c),$ for all indices $i, j \in \overline{0, N_{2}-1}, \; $ such that $ u_{i}, u_{j}\leq \tilde{u}(c).$ By introducing the new discretization points $\bar{u}_{i}$ we aim to simplify some formulas coming later in this section.}
\end{equation}
  \item \label{condition_three} functions $\tilde{u}(x)$ and $\tilde{x}(u)$ satisfy the boundary conditions
\begin{equation}\label{discretized_equations_boundaryConditions}
  \tilde{u}(0) = 0, \; \tilde{x}(u_{b}) = b.
\end{equation}
and the "matching" conditions
\begin{equation}\label{matching_conditions_u}
  \tilde{x}(\tilde{u}(c)) = c,\;
  \tilde{x}^{\prime}(\tilde{u}(c)) = \frac{1}{\tilde{u}^{\prime}(c)}.
\end{equation}
\end{enumerate}

\begin{lemma}\label{lemma_about_inverse_to_x_tilde}
Let $\tilde{u}(x)$ and $\tilde{x}(u)$ be a pair of functions, satisfying conditions $\ref{condition_one}, \ref{condition_two}, \ref{condition_three}.$ Then the inverse function $\tilde{x}^{-1}(x)$ exists on $[c, b]$ and belongs to $C^{1}([c, b]).$
\end{lemma}
\begin{proof}
Indeed, from \eqref{discretized_equation_inverse}, \eqref{discretized_equation_inverse_term_definition} it follows that
\begin{equation}\label{explicite_derivative_of_x}
  \tilde{x}^{\prime}(u) = \frac{\tilde{u}^{\prime}(c)}{\sqrt{1 - 2\left(\tilde{u}^{\prime}(c)\right)^{2}\int\limits_{\bar{u}_{i}}^{u}\beta(\mathbb{P}_{u}[\tilde{x}^{\prime}(\xi)], \mathbb{P}_{u}[\tilde{x}(\xi)], \xi)d\xi}}.
\end{equation}
The fact that both $\tilde{u}(x)$ and $\tilde{x}(u)$ are continuously differentiable (each on its domain) and condition \eqref{matching_conditions_u} yield us
$$\tilde{u}^{\prime}(c) \neq 0,$$
which, in the light of formula \eqref{explicite_derivative_of_x}, guarantees that $\tilde{x}(u)$ is monotone on $[\tilde{u}(c), u_{b}].$
\end{proof}

\begin{theorem}\label{Existence_theorem}
Let conditions \eqref{lema_1_positiveness_of_nonlinearity} and
\begin{equation}\label{positiveness_of_N_derivative_with_respect_to_x}
  N^{\prime}_{x}(u, x), N^{\prime}_{u}(u, x), N^{\prime}_{xu}(u, x), N^{\prime}_{uu}(u, x) \geq 0,\; \forall u\in [0, u_{b}], \; \forall x\in [a,b]
\end{equation}
 hold true. Then for an arbitrary fixed $c$ from $(a, b),$ there exists a pair of functions $\tilde{u}(x),$ $\tilde{x}(u)$ satisfying conditions $\ref{condition_one}, \ref{condition_two}, \ref{condition_three}$ and the following inequalities:
 \begin{equation}\label{boundaries_for_u_tilde}
0 < \tilde{u}(x) < u_{b},\; 0 < \tilde{u}^{\prime}(x),\; 0 \leq \tilde{u}^{\prime\prime}(x),\; \forall x\in (a,c],
\end{equation}
\begin{equation}\label{boundaries_for_x_tilde}
c < \tilde{x}(u) < b,\; \tilde{x}^{\prime}(u) > 0,\; \tilde{x}^{\prime\prime}(u) \leq 0,\; \forall u\in (\tilde{u}(c),u_{b}].
\end{equation}
\end{theorem}

In order to prove Theorem \ref{Existence_theorem} we first need to prove a few auxiliary statements below.

\begin{lemma}\label{lemma_about_monotonicity_of_u_nu}
Let $\tilde{u}_{\nu}(x) \in C^{1}([a, c])$ denote the solution to equation \eqref{discretized_equation_straight}, \eqref{discretized_equation_straight_term_definition} subjected to initial conditions
\begin{equation}\label{auxiliary_problem_init_conditions}
  \tilde{u}_{\nu}(0) = 0,\; \tilde{u}_{\nu}^{\prime}(0) = \nu
\end{equation}
and let conditions \eqref{lema_1_positiveness_of_nonlinearity} and \eqref{positiveness_of_N_derivative_with_respect_to_x} hold true.
Then $\forall \nu, \bar{\nu} \in [0, +\infty),$ $\forall x\in [a, c]$
\begin{equation}\label{auxiliary_problem_monotonicity_function}
  \tilde{u}_{\nu}(x) > \tilde{u}_{\bar{\nu}}(x),
\end{equation}
\begin{equation}\label{auxiliary_problem_monotonicity_derivative}
  \tilde{u}_{\nu}^{\prime}(x) > \tilde{u}_{\bar{\nu}}^{\prime}(x),
\end{equation}
provided that
\begin{equation}\label{auxiliary_problem_inequality_for_initial_derivatives}
  \nu > \bar{\nu}.
\end{equation}
\end{lemma}
\begin{proof}
Let $\tilde{u}_{\nu, \mu, i}(x) \in C^{1}([a,c])$ denote the solution of equation \eqref{discretized_equation_straight}, \eqref{discretized_equation_straight_term_definition} subjected to initial conditions
\begin{equation}\label{auxiliary_problem_init_conditions_extended}
  \tilde{u}(x_{i}) = \mu,\; \tilde{u}^{\prime}(x_{i}) = \nu, \; \mu \geq 0,\; \nu > 0,\;\forall i\in \overline{0, N_{1} - 1},
\end{equation}
so that $$\tilde{u}_{\nu}(x) \equiv \tilde{u}_{\nu, 0, 0}(x).$$

Let us fix some arbitrary $j\in \overline {0, N_{1}-1}$ and assume that
\begin{equation}\label{auxiliary_problem_assumption_for_mu}
  \tilde{u}_{\nu_{j}, \mu_{j}, j}(x_{j}) \geq \tilde{u}_{\bar{\nu}_{j}, \bar{\mu}_{j}, j}(x_{j}) \geq 0.
\end{equation}
\begin{equation}\label{auxiliary_problem_assumption_for_nu}
  \tilde{u}_{\nu_{j}, \mu_{j}, j}^{\prime}(x_{j})>\tilde{u}_{\bar{\nu}_{j}, \bar{\mu}_{j}, j}^{\prime}(x_{j}) > 0,
\end{equation}
Under the conditions of the lemma and assumptions \eqref{auxiliary_problem_assumption_for_nu}, \eqref{auxiliary_problem_assumption_for_mu} we are going to prove that
\begin{equation}\label{auxiliary_problem_monotonicity_function_partial}
  \tilde{u}_{\nu_{j}, \mu_{j}, j}(x) > \tilde{u}_{\bar{\nu}_{j}, \bar{\mu}_{j}, j}(x),\; \forall x\in(x_{j}, x_{j+1}],
\end{equation}
\begin{equation}\label{auxiliary_problem_monotonicity_derivative_partial}
  \tilde{u}_{\nu_{j}, \mu_{j}, j}^{\prime}(x) > \tilde{u}_{\bar{\nu}_{j}, \bar{\mu}_{j}, j}^{\prime}(x),\; \forall x\in(x_{j}, x_{j+1}].
\end{equation}
By definition, functions $\tilde{u}_{\nu_{j}, \mu_{j}, j}(x),$ $\tilde{u}_{\bar{\nu}_{j}, \bar{\mu}_{j}, j}(x)$ satisfy equations
\begin{equation}\label{auxiliary_problem_equation_partial_first}
  \tilde{u}_{\nu_{j}, \mu_{j}, j}^{\prime\prime}(x) - \alpha_{j}(\nu_{j}, \mu_{j},x)\tilde{u}_{\nu_{j}, \mu_{j}, j}(x) = 0,\; \forall x\in[x_{j}, x_{j+1}],
\end{equation}
\begin{equation}\label{auxiliary_problem_equation_partial_second}
  \tilde{u}_{\bar{\nu}_{j}, \bar{\mu}_{j}, j}^{\prime\prime}(x) - \alpha_{j}(\bar{\nu}_{j}, \bar{\mu}_{j},x)\tilde{u}_{\bar{\nu}_{j}, \bar{\mu}_{j}, j}(x) = 0,\; \forall x\in[x_{j}, x_{j+1}]
\end{equation}
respectively.
It is easy to verify, that under conditions \eqref{lema_1_positiveness_of_nonlinearity}, \eqref{positiveness_of_N_derivative_with_respect_to_x} the inequality
\begin{equation}\label{auxiliary_problem_inequalities_for_alpha_partial}
  \alpha_{j}(\nu_{j}, \mu_{j},x) \geq \alpha_{j}(\bar{\nu}_{j}, \bar{\mu}_{j},x) \geq 0,\; \forall x\in (x_{j}, x_{j+1}]
\end{equation}
holds true.

Subtracting \eqref{auxiliary_problem_equation_partial_second} from \eqref{auxiliary_problem_equation_partial_first} and using inequalities \eqref{auxiliary_problem_inequalities_for_alpha_partial}, we get the estimate
\begin{equation}\label{auxiliary_problem_estimate_for_dofference_partial}
  w^{\prime\prime}(x) - \alpha_{j}(\nu_{j}, \mu_{j},x) w(x) \geq 0,\; w(x) = \tilde{u}_{\nu_{j}, \mu_{j}, j}(x) - \tilde{u}_{\bar{\nu}_{j}, \bar{\mu}_{j}, j}(x),\; \forall x\in [x_{j}, x_{j+1}].
\end{equation}
From \eqref{auxiliary_problem_assumption_for_mu} and \eqref{auxiliary_problem_assumption_for_nu} it follows that
$$w(x_{j}) \geq 0,\; w^{\prime}(x_{j}) > 0,$$
which, in conjunction with the {\it maximum principle} (see, for example, \cite[Theorems 3, 4, p. 6--7]{Protter_Weinberger_Max_principle})), yields us the inequality
$$w^{\prime}(x) > 0,\; \forall x\in [x_{j}, x_{j+1}].$$ The latter automatically implies inequalities \eqref{auxiliary_problem_monotonicity_function_partial} \eqref{auxiliary_problem_monotonicity_derivative_partial}.

By now we proved that if conditions \eqref{auxiliary_problem_assumption_for_mu}, \eqref{auxiliary_problem_assumption_for_nu} hold true for some $j\in \overline{0, N_{1}-1}$ then (under the conditions of the lemma) they are also fulfilled for $j+1$ with
$$\nu_{j+1} = \tilde{u}^{\prime}_{\nu_{j}, \mu_{j}, j}(x_{j+1}),\; \mu_{j+1} = \tilde{u}_{\nu_{j}, \mu_{j}, j}(x_{j+1})$$
As it can be easily seen, inequality \eqref{auxiliary_problem_inequality_for_initial_derivatives} implies conditions \eqref{auxiliary_problem_assumption_for_mu}, \eqref{auxiliary_problem_assumption_for_nu} for $j=0$ with
$$\nu_{0} = \nu,\; \mu_{0} = 0,$$
$$\bar{\nu}_{0} = \bar{\nu},\; \bar{\mu}_{0} = 0$$
and the lemma's statement obviously follows from what was proved above and the principle of mathematical induction.
\end{proof}

\begin{lemma}\label{lemma_about_continuity_of_u_with_respect_to_nu}
Let the conditions of Lemma \ref{lemma_about_monotonicity_of_u_nu} hold true. Then $\tilde{u}_{\nu}(x)$ and $\tilde{u}^{\prime}_{\nu}(x),$ as functions of parameter $\nu,$ are continuous on $[0, +\infty),$ $\forall x\in [a,c].$
\end{lemma}
\begin{proof}
The statement of the lemma almost immediately follows from the corresponding theorem about continuity of solutions of IVPs with respect to initial conditions and parameters (see, for example, \cite[Theorem 8.40, p 372]{Kelley_Peterson_2010}).
\end{proof}

\begin{lemma}\label{lemma_about_a_point_nu_Ast}
  Let the conditions of Lemma \ref{lemma_about_monotonicity_of_u_nu} hold true. Then there exists a unique value $\nu^{\ast} > 0$ such that
  \begin{equation}\label{property_of_nu_ast}
    \tilde{u}_{\nu^{\ast}}(c) = u_{b}.
  \end{equation}
\end{lemma}
\begin{proof}
From conditions \eqref{lema_1_positiveness_of_nonlinearity}, \eqref{positiveness_of_N_derivative_with_respect_to_x} and the maximum principle it follows that
$$\tilde{u}_{\nu}(x) > \nu (x-a),\; \forall x\in [a,c], \; \forall \nu > 0.$$
The latter yields us the inequality
$$\tilde{u}_{\nu}(c) > u_{b}$$
provided that
$$\nu \geq \frac{u_{b}}{c-a},$$ which, in conjunction with the obvious equality $$\tilde{u}_{0}(c) = 0,$$ Lemma \ref{lemma_about_continuity_of_u_with_respect_to_nu} and the {\it Bolzano's theorem,} provides us the existence of
$\nu^{\ast}$ mentioned in the Lemma. The uniqueness follows from the monotonicity properties of $\tilde{u}_{\nu}(x)$ as a function of parameter $\nu$ (Lemma \ref{lemma_about_monotonicity_of_u_nu}).
\end{proof}

\begin{lemma}\label{lemma_x_tend_to_infty_when_nu_tend_to_zero}
Let the conditions of Lemma \ref{lemma_about_monotonicity_of_u_nu} hold true and let $\tilde{x}_{\nu}(u)\in C^{1}([\tilde{u}_{\nu}(c), u_{b}])$ denote the solution to equation \eqref{discretized_equation_inverse}, \eqref{discretized_equation_inverse_term_definition} subjected to initial conditions
\begin{equation}\label{lemma_x_tend_to_infty_matching_conditions_with_nu}
  \tilde{x}_{\nu}(\tilde{u}_{\nu}(c)) = c,\; \tilde{x}^{\prime}_{\nu}(\tilde{u}_{\nu}(c)) = \frac{1}{\tilde{u}^{\prime}_{\nu}(c)},\; \nu\in [0, \nu^{\ast}]
\end{equation}
where $\nu^{\ast}$ was introduced in Lemma \ref{lemma_about_a_point_nu_Ast}.
Then $\phi(\nu) = \tilde{x}_{\nu}(u_{b})$ is a continuous function of $\nu\in (0, \nu^{\ast})$ and
\begin{equation}\label{limit_of_x_when_nu_tends_to_nu_ast}
  \lim\limits_{\nu \uparrow \nu^{\ast}} \phi(\nu) = c.
\end{equation}
Additionally to that, there exists $\nu_{\ast} \in (0, \nu^{\ast}),$ such that
\begin{equation}\label{limit_of_x_when_nu_tends_to_infty}
  \phi(\nu_{\ast}) > b.
\end{equation}
\end{lemma}

\begin{proof}
  We start by proving that the function $\phi(\nu)$ is continuous on $(0, \nu^{\ast}).$

  It is easy to see that on each interval $[\bar{u}_{i}, \bar{u}_{i+1}],$ $i \in \overline{0, N_{2}-1}$ function $\tilde{x}_{\nu}(u)$ can be expressed in a recursive way
  \begin{equation}\label{expression_for_tilde_x_on_suninterval}
    \tilde{x}_{\nu}(u) = \tilde{x}_{\nu, i}(u) = \int\limits_{\bar{u}_{i}}^{u}\frac{\tilde{x}_{\nu, i-1}(\bar{u}_{i})d\eta}{\sqrt{1-2\left(\tilde{x}_{\nu, i-1}(\bar{u}_{i})\right)^{2}\int\limits_{\bar{u}_{i}}^{\eta}\beta_{i}\left(\tilde{x}^{\prime}_{\nu, i-1}(\bar{u}_{i}), \tilde{x}_{\nu, i-1}(\bar{u}_{i}), \xi\right)d\xi}} + \tilde{x}_{\nu, i-1}(\bar{u}_{i}),
  \end{equation}
  where
  \begin{equation}\label{value_for_bar_u_0}
    \bar{u}_{0} = \tilde{u}_{\nu}(c),\; \tilde{x}_{\nu, -1}(\bar{u}_{0}) = c,\; \tilde{x}^{\prime}_{\nu, -1}(\bar{u}_{0}) = \frac{1}{\tilde{u}_{\nu}^{\prime}(c)}.
  \end{equation}

  According to the definition of $\bar{u}_{i}$ given in \eqref{definition_of_u_bar}, some intervals $[\bar{u}_{i}, \bar{u}_{i+1}]$ have  zero measure, containing a single point $\tilde{u}_{\nu}(c).$ This, however, does not affect the correctness of the reasoning below.

  From \eqref{expression_for_tilde_x_on_suninterval} it follows that
  \begin{eqnarray}\label{recursive_formulas_for_x_tilde_nu}
    \tilde{x}_{\nu, i}(\bar{u}_{i+1}) &=& \phi_{i}(\tilde{x}^{\prime}_{\nu, i-1}(\bar{u}_{i}), \tilde{x}_{\nu, i-1}(\bar{u}_{i}), \bar{u}_{i}), \nonumber \\
    \tilde{x}^{\prime}_{\nu, i}(\bar{u}_{i+1}) &=& \psi_{i}(\tilde{x}^{\prime}_{\nu, i-1}(\bar{u}_{i}), \tilde{x}_{\nu, i-1}(\bar{u}_{i}), \bar{u}_{i}), \\
  \end{eqnarray}
where $\bar{u}_{i} = \bar{u}_{i}(\tilde{u}_{\nu}(c))$ (see \eqref{definition_of_u_bar}),
\begin{equation}\label{function_phi_i}
  \phi_{i}(x^{\prime}, x, u) = \int\limits_{u}^{\bar{u}_{i+1}}\frac{x^{\prime}d\eta}{\sqrt{1-2\left(x^{\prime}\right)^{2}\int\limits_{u}^{\eta}\beta_{i}\left(x^{\prime}, x, \xi\right)d\xi}} + x,
\end{equation}

\begin{equation}\label{function_psi_i}
  \psi_{i}(x^{\prime}, x, u) = \frac{x^{\prime}}{\sqrt{1-2\left(x^{\prime}\right)^{2}\int\limits_{u}^{\bar{u}_{i+1}}\beta_{i}\left(x^{\prime}, x, \xi\right)d\xi}},
\end{equation}
$$u \in [0, \bar{u}_{i+1}],\; x \in [a, b],\; x^{\prime} \in \mathbb{R}, \;i\in \overline{0, N_{2}-1}.$$

It is easy to see that (under conditions \eqref{lema_1_positiveness_of_nonlinearity}, \eqref{positiveness_of_N_derivative_with_respect_to_x}) functions \eqref{function_phi_i}, \eqref{function_psi_i} are continuous on their domains, which, in conjunction with the recursive formulas \eqref{recursive_formulas_for_x_tilde_nu} and initial conditions \eqref{value_for_bar_u_0}, implies that $\tilde{x}_{\nu}(u_{b})$ is continuously dependent on $\tilde{u}_{\nu}^{\prime}(c),$ $\tilde{u}_{\nu}(c).$ On the other hand, according to Lemma \ref{lemma_about_continuity_of_u_with_respect_to_nu} the latter two quantities are continuous functions of the parameter $\nu,$ which completes the first part of the proof.

To prove equality \eqref{limit_of_x_when_nu_tends_to_nu_ast} we can, without loss of generality, to assume that $\nu < \nu^{\ast}$ is so close to $\nu^{\ast}$ that
$$\bar{u}_{N_{2}-1} \leq \tilde{u}_{\nu}(c) < \bar{u}_{N_{2}} = u_{b}.$$ This allows us to reduce the limit in the left hand side of \eqref{limit_of_x_when_nu_tends_to_nu_ast} to the following form
$$
\lim\limits_{\nu \uparrow \nu^{\ast}}\tilde{x}_{\nu}(u_{b}) = c + \lim\limits_{\nu \uparrow \nu^{\ast}}\int\limits_{\tilde{u}_{\nu}(c)}^{u_{b}}\frac{d\eta}{\sqrt{\left(\tilde{u}_{\nu}^{\prime}(c)\right)^{2} -2\int\limits_{\tilde{u}_{\nu}(c)}^{\eta}\beta_{N_{2}-1}(1/\tilde{u}_{\nu}^{\prime}(c), c, \tilde{u}_{\nu}(c))d\xi}}.
$$
Since $\tilde{u}_{\nu}(c)$ tends to $u_{b}$ as $\nu$ tends to $\nu^{\ast}$ (see Lemma \ref{lemma_about_a_point_nu_Ast}), the limit in the right hand side of the equality above is equal to $0$, which proofs the target equality \eqref{limit_of_x_when_nu_tends_to_nu_ast}.

Finally, we focus on proving the existence of $\nu_{\ast}$ mentioned in the lemma. To do so, let us estimate $\tilde{x}_{\nu}(u)$ from below:
$$\tilde{x}_{\nu}(u) - c = \int\limits_{\tilde{u}_{\nu}(c)}^{u}\left[\left(\tilde{u}^{\prime}_{\nu}(c)\right)^{2} - 2\int\limits_{\tilde{u}_{\nu}(c)}^{\eta}\beta\left(\mathbb{P}_{u}(\tilde{x}^{\prime}_{\eta}(\xi)), \mathbb{P}_{u}(\tilde{x}_{\nu}(\xi)), \xi\right)d\xi\right]^{-\frac{1}{2}}d\eta $$
$$\geq\int\limits_{\tilde{u}_{\nu}(c)}^{u}\left[\left(\tilde{u}^{\prime}_{\nu}(c)\right)^{2} + 2\int\limits_{\tilde{u}_{\nu}(c)}^{\eta}(T\xi + M(\xi-\tilde{u}_{\nu}(c)))d\xi\right]^{-\frac{1}{2}}d\eta $$
$$ =\int\limits_{\tilde{u}_{\nu}(c)}^{u}\left[\left(\tilde{u}^{\prime}_{\nu}(c)\right)^{2} + T\left(\eta + \tilde{u}_{\nu}(c)\right)\left(\eta - \tilde{u}_{\nu}(c)\right) + M\left(\eta-\tilde{u}_{\nu}(c)\right)^{2}\right]^{-\frac{1}{2}}d\eta $$
$$\geq \int\limits_{\tilde{u}_{\nu}(c)}^{u}\left[\left(\tilde{u}^{\prime}_{\nu}(c)\right)^{2} + \left(\frac{T}{2\sqrt{M}}\left(\eta + \tilde{u}_{\nu}(c)\right) + \sqrt{M}\left(\eta - \tilde{u}_{\nu}(c)\right)\right)^{2}\right]^{-\frac{1}{2}}d\eta$$
$$ = \left.\frac{1}{Q}\ln\left(Q\eta + R\tilde{u}_{\nu}(c) + \sqrt{\left(\tilde{u}^{\prime}_{\nu}(c)\right)^{2} + \left(Q\eta + R\tilde{u}_{\nu}(c)\right)^{2}}\right)\right|_{\eta = \tilde{u}_{\nu}(c)}^{\eta =u}$$
$$ = \frac{1}{Q}\ln\left(\frac{Qu + R\tilde{u}_{\nu}(c) + \sqrt{\left(\tilde{u}^{\prime}_{\nu}(c)\right)^{2} + \left(Qu + R\tilde{u}_{\nu}(c)\right)^{2}}}{\frac{T}{\sqrt{M}}\tilde{u}_{\nu}(c) + \sqrt{(\tilde{u}_{\nu}^{\prime}(c))^{2} + \left(\frac{T}{\sqrt{M}}\tilde{u}_{\nu}(c)\right)^{2}}}\right)$$
$$\geq \frac{1}{Q}\ln\left(\frac{\frac{T}{\sqrt{M}}u}{\frac{T}{\sqrt{M}}\tilde{u}_{\nu}(c) + \sqrt{(\tilde{u}_{\nu}^{\prime}(c))^{2} + \left(\frac{T}{\sqrt{M}}\tilde{u}_{\nu}(c)\right)^{2}}}\right) $$
\begin{equation}\label{inequality_to_estimate_u_prime_c_from_below}
  \geq\frac{1}{Q}\ln\left(\frac{\frac{T}{\sqrt{M}}u_{b}}{\tilde{u}_{\nu}^{\prime}(c)\left(\frac{T}{\sqrt{M}}(c-a) + \sqrt{1 + \left(\frac{T}{\sqrt{M}}(c-a)\right)^{2}}\right)}\right),
\end{equation}

where
$$T = \max\left\{N(u,x)\;|\; x\in [c, b],\; u\in [0, u_{b}]\right\}, $$
$$M = \max\left\{1, \max\left\{\mathcal{N}^{\prime}_{u}(u, x) + N^{\prime}_{x}(u, x)(b-a)\;|\; x\in [c, b],\; u\in [0, u_{b}]\right\}\right\}\footnote{Here we used the inequality
$$b-a \geq \tilde{x}_{\nu}(\bar{u}_{i}) - c + c - a \geq (\bar{u}_{i} - \bar{u}_{0})\tilde{x}_{\nu}^{\prime}(\bar{u}_{i}) + (\bar{u}_{0} - 0)\tilde{x}^{\prime}_{\nu}(\bar{u}_{0}) \geq \bar{u}_{i}\tilde{x}_{\nu}^{\prime}(\bar{u}_{i}),$$
that follows from the {\it mean value theorem} and the concavity of function $\tilde{x}_{\nu}(u).$}$$
$$Q = \frac{T+2M}{2\sqrt{M}},\; R = \frac{T-2M}{2\sqrt{M}}.$$

From estimate \eqref{inequality_to_estimate_u_prime_c_from_below} we get that
$$\tilde{x}_{\nu}(u_{b}) > b$$
provided that
\begin{equation}\label{estimate_from_above_for_u_peimw_c}
  0 \leq \tilde{u}_{\nu}^{\prime}(c) < \tilde{u}^{\prime}_{\ast} \stackrel{def}{=} \frac{ \exp\left(Q(c-b)\right) \frac{T}{\sqrt{M}}u_{b}}{\left(\frac{T}{\sqrt{M}}(c-a) + \sqrt{1 + \left(\frac{T}{\sqrt{M}}(c-a)\right)^{2}}\right)}.
\end{equation}
According to Lemma \ref{lemma_about_continuity_of_u_with_respect_to_nu}, the latter inequality is satisfied for sufficiently small values of parameter $\nu > 0.$ This yields the existence of $\nu_{\ast}$ mentioned in the statement of the lemma.
\end{proof}

\begin{proof}[Proof of Theorem \ref{Existence_theorem}]
From Lemma \ref{lemma_x_tend_to_infty_when_nu_tend_to_zero} it follows that  function
$$f(\nu) = \tilde{x}_{\nu}(u_{b}) - b,$$
is continuous on $[\nu_{\ast}, \nu^{\ast}]$ and takes different signs in the endpoints of the interval. The latter, according to the {\it Bolzano's intermediate value theorem}, implies the existence of $\check{\nu} \in (\nu_{\ast}, \nu^{\ast}),$ such that $f(\check{\nu}) = 0.$ In the way described above, $\check{\nu}$ (which might not be unique) uniquely defines a pair of functions $\tilde{u}(x) = \tilde{u}(\check{\nu}, x),$ $\tilde{x}(u) = \tilde{x}(\check{\nu}, u)$
mentioned in Theorem \ref{Existence_theorem}. Inequalities \eqref{boundaries_for_u_tilde} and \eqref{boundaries_for_x_tilde} almost obviously follow from the monotonicity of $\tilde{u}(\check{\nu}, x)$ with respect to $x$ (see Lemma \ref{lemma_about_monotonicity_of_u_nu}, when $\bar{\nu} = 0$) and the monotonicity of $\tilde{x}(\check{\nu}, u)$ with respect to $u$ (see Lemma \ref{lemma_about_inverse_to_x_tilde}) respectively.
\end{proof}

\begin{remark}\label{remark_about_u_tilde_prime_c_bounded from below}
In scope of Lemma \ref{lemma_x_tend_to_infty_when_nu_tend_to_zero} it was proved that if functions $\tilde{u}(x)$ and $\tilde{x}(u)$ satisfy conditions \eqref{discretized_equation_straight}---\eqref{matching_conditions_u}, then $\tilde{u}^{\prime}(c)$ is bounded from below (see inequality \eqref{estimate_from_above_for_u_peimw_c}) by a constant $\tilde{u}^{\prime}_{\ast},$ depending on the function $N(u,x)$ and parameters $a, b, c, u_{b}.$
\end{remark}

It is also not difficult to prove a similar estimate for $\tilde{u}^{\prime}(c)$ as stated in the lemma below.
\begin{lemma}\label{lemma_about_u_tilde_prime_c_bounded_from_above}
Let functions $\tilde{u}(x)$ and $\tilde{x}(u)$ satisfy conditions \eqref{discretized_equation_straight} -- \eqref{matching_conditions_u}, then $\tilde{u}^{\prime}(c)$ is bounded from above
\begin{equation}\label{u_tilde_prime_c_bounded_from_above}
  \tilde{u}^{\prime}(c) \leq \tilde{u}^{\prime\ast} \stackrel{def}{=} \frac{u_{b}}{b-c}.
\end{equation}
\end{lemma}

The existence of constant $\tilde{u}^{\prime}_{\ast}$ imposes a restriction from below on the value of $\tilde{u}^{\prime}(u)$ as it is stated by the following lemma.
\begin{lemma}
  Let functions $\tilde{u}(x)$ and $\tilde{x}(u)$ satisfy conditions \eqref{discretized_equation_straight} -- \eqref{matching_conditions_u} and inequalities \eqref{lema_1_positiveness_of_nonlinearity}, \eqref{positiveness_of_N_derivative_with_respect_to_x} hold true. Then
  \begin{equation}\label{restriction_on_tilde_u_prime_from_below}
    \tilde{u}^{\prime}_{\ast\ast} \stackrel{def}{=} \tilde{u}^{\prime}_{\ast}\exp\left(- \frac{1}{2}c^{2}\left(N^{\prime}_{u}(u_{b}, c)\tilde{u}^{\prime\ast} + \max\limits_{0\leq x \leq c}N^{\prime}_{x}(u_{b}, x)\right) - c N(u_{b}, c)\right) \leq \tilde{u}^{\prime}(u)
  \end{equation}
  where constants $\tilde{u}^{\prime}_{\ast}$ and $\tilde{u}^{\prime\ast}$ are defined in \eqref{estimate_from_above_for_u_peimw_c} and \eqref{u_tilde_prime_c_bounded_from_above} respectively.
\end{lemma}
\begin{proof}
Using the Gronwall's inequality (see, for example, \cite[p. 42]{Teschl_ODE_and_DS}) and estimate \eqref{estimate_from_above_for_u_peimw_c} we get
$$\tilde{u}^{\prime}_{\ast} \leq \tilde{u}^{\prime}(c) \leq \tilde{u}^{\prime}(u)\exp\left(\int\limits_{u}^{c}\alpha(\mathbb{P}_{x}[\tilde{u}^{\prime}(\xi)], \mathbb{P}_{x}[\tilde{u}(\xi)], \xi)d\xi\right) \leq \frac{\tilde{u}^{\prime}(u)\tilde{u}^{\prime}_{\ast}}{\tilde{u}^{\prime}_{\ast\ast}}.$$
Dividing the inequality above by $\tilde{u}^{\prime}_{\ast}/\tilde{u}^{\prime}_{\ast\ast}$, we get estimate \eqref{restriction_on_tilde_u_prime_from_below}.

\end{proof}

\section{Error analysis}\label{section_error_analysis}

The current section is focused on the approximation properties of the numerical scheme introduced above. In particular, Theorem \ref{Main_theorem_abour_apptoximation_prorerties_of_the_SI_method} answers the question about interconnection between functions $\tilde{u}(x), \tilde{x}(u)$ (satisfying conditions \eqref{discretized_equation_straight} -- \eqref{matching_conditions_u}) and the solutions $u(x)$ and $x(u)$ respectively. To prove the theorem we first need to justify a few auxiliary statements that follow below.



\begin{theorem}\label{aux_theorem_about_approx_properties}

Let conditions \eqref{lema_1_positiveness_of_nonlinearity} and \eqref{positiveness_of_N_derivative_with_respect_to_x} hold true and $\tilde{u}(x) = \tilde{u}(x, h),$ $\tilde{x}(u) = \tilde{x}(u, h)$ is a pair of functions mentioned in Theorem \ref{Existence_theorem} \textup{(}whose conditions are obviously fulfilled\textup{)}, where

\begin{equation}\label{definition_of_h}
  h = \max\left\{\max\limits_{i\in\overline{1, N_{1}}}(x_{i}-x_{i-1}), \max\limits_{i\in\overline{1, N_{2}}}(\bar{u}_{i}-\bar{u}_{i-1})\right\} > 0.
\end{equation}
Then, for $h$ sufficiently small, there exists a function $\hat{x}(u) = \hat{x}(u, h)\in C^{2}([0, u_{b}]),$
which satisfies equation \eqref{Intro_Equation_Inverse} subjected to initial conditions
\begin{equation}\label{init_conditions_for_u_hat}
  \hat{x}(0) = \tilde{u}^{-1}(0) = a,\; \hat{x}^{\prime}(0) = \frac{1}{\tilde{u}^{\prime}(a, h)},
\end{equation}
and the following estimates hold true:
\begin{equation}\label{approx_estimates_for_u_hat}
  \|\tilde{u}(x) - \hat{u}(x)\|_{ [a, c], 1} \leq \kappa_{1} h^{2},
\end{equation}
\begin{equation}\label{approx_estimates_for_x_hat}
  \|\tilde{x}(u) - \hat{x}(u)\|_{[\tilde{u}(c), u_{b}], 1} \leq \kappa_{2} h^{2},
\end{equation}
where
$$\hat{u}(x) \stackrel{def}{=} \hat{x}^{-1}(x),$$
$$\|f(\xi)\|_{[\xi_{1}, \xi_{2}], i} \stackrel{def}{=} \max\limits_{\xi\in [\xi_{1}, \xi_{2}]}\left\{|f^{(0)}(\xi)|,\ldots, |f^{(i)}(\xi)|\right\},\; f^{(k)}(\xi)\stackrel{def}{=} \frac{d^{k}}{d\xi^{k}}f(x)$$
and constants $\kappa_{1}, \kappa_{2} > 0$ depend on BVP \eqref{Intro_Equation}, \eqref{Intro_boundary_conditions} only.
\end{theorem}
\begin{proof}
  Logically, we consider the proof to consist of 3 parts, so that each subsequent part relies on the results of the previous ones. For the convenience of the reader, we make this division explicit by adding the corresponding headers.

{\it Part 1 : existence of $\hat{x}(u)$ on $[0, \tilde{u}(c)]$ and estimate \eqref{approx_estimates_for_u_hat}. }

    Let us, for a moment, step back from the notations of the theorem and re-define function $\hat{u}(x)$ to be the solution of equation \eqref{Intro_Equation} subjected to initial conditions
    \begin{equation}\label{theorem_about_approximation_init_conditions_for_aux_problem}
      u(a) = 0,\; u^{\prime}(a) = \tilde{u}(a).
    \end{equation}
    If $\hat{u}(x)$ (defined in such a way) exists on $[a, c]$ and estimates \eqref{approx_estimates_for_u_hat} holds true, then we can be sure that, for $h$ sufficiently small, $\hat{u}(c) > 0$ (since $\tilde{u}(c) > 0,$ see \eqref{c_inequality_for_straight_discretization}). The latter, in conjunction with Lemma \ref{lemma_monotonicity_of_u} (whose conditions are fulfilled), yields us existence of function $\hat{u}^{-1}(u)$ on $[0, \hat{u}(c)],$ which, apparently, can be taken for $\hat{x}(u).$ If, additionally, we manage to prove that $\hat{u}(x)$ exists on a little bit bigger interval, say $[a, c+\delta],$ for some $\delta > 0$ independent on $h,$ then, taking into account monotonicity of $\hat{u}(x),$ and restricting $h$ even more (if needed), we can ensure that $\tilde{u}(c) < \hat{u}(c+\delta),$ and thus get the existence of $\hat{x}(u)$ on $[0, \tilde{u}(c)].$ With this scheme in mind, we proceed by proving the existence of $\hat{u}(x)$ and estimate \eqref{approx_estimates_for_u_hat}.

    Let us fix some arbitrary $\varepsilon > 0.$

    Rewriting IVP \eqref{Intro_Equation}, \eqref{theorem_about_approximation_init_conditions_for_aux_problem} in an equivalent vector form
    \begin{equation}\label{theorem_about_approximation_aux_problem_vector_form}
      \dot{\mathbf{u}}(x) \stackrel{def}{=} \left[
        \begin{array}{c}
          u^{\prime\prime}(x) \\
          u^{\prime}(x) \\
        \end{array}
      \right] = \mathbf{F}(\mathbf{u}(x), x) \stackrel{def}{=} \left[
                  \begin{array}{c}
                    \mathcal{N}(u(x), x) \\
                    u^{\prime}(x) \\
                  \end{array}
                \right],\; \mathbf{u}(a) = \left[
                                       \begin{array}{c}
                                         \tilde{u}^{\prime}(a) \\
                                         0 \\
                                       \end{array}
                                     \right],
    \end{equation}
     and applying the {\it Picard-Lindelof Theorem} (see, for example, \cite[p. 350]{Kelley_Peterson_2010}) to it, we conclude that solution $\hat{u}(x)$ exists at least on $$\left[a, c_{1}\right],\; c_{1} = \min\left\{a+\frac{\varepsilon}{2M}, c\right\}$$
     where
     $$
     M = \max\limits_{(\mathbf{u},x) \in \mathbb{D}_{\varepsilon}}\left\| \mathbf{F}(\mathbf{u}, x)\right\|,\; \mathbb{D}_{\varepsilon} = \mathbb{D}_{\varepsilon, u} \times \mathbb{D}_{\varepsilon, u^{\prime}}\times [a,c].
     $$
     $$\mathbb{D}_{\varepsilon, u} = \left\{u\in \mathbb{R} \; | \; -\varepsilon \leq u \leq u_{b} + \varepsilon \right\} \supseteq \left\{u \in \mathbb{R}\; | \; \min\limits_{x\in[a,c]}\tilde{u}(x) - \varepsilon \leq u \leq \max\limits_{x\in[a,c]}\tilde{u}(x) + \varepsilon \right\}\footnote{See estimate \eqref{boundaries_for_u_tilde}.},$$
     $$\mathbb{D}_{\varepsilon, u^{\prime}} = \left\{u\in \mathbb{R} \; | \; -\varepsilon \leq u \leq \tilde{u}^{\prime\ast} + \varepsilon\right\} \supseteq \left\{u \in \mathbb{R}\; | \; \min\limits_{x\in[a,c]}\tilde{u}^{\prime}(x) - \varepsilon \leq u \leq \max\limits_{x\in[a,c]}\tilde{u}^{\prime}(x) + \varepsilon\right\}\footnote{See estimate \eqref{u_tilde_prime_c_bounded_from_above}.},$$
     and, in addition to that,
     \begin{equation}\label{estimate_for_u_from_PL_theorem}
       \|\mathbf{u}(x) - \mathbf{u}(a)\| \leq \frac{\varepsilon}{2},\; \forall x\in [a, c_{1}].
     \end{equation}

To simplify the proof, we assume that $$h < c_{1} - a.$$
The latter, guarantees, that the set $$\mathbb{I}_{1} = \left\{i\in \overline{1, N_{1}}\; | \; x_{i} < c_{1}\right\}$$ is non-empty.

   It is easy to see, that on each interval $[x_{i-1}, x_{i}],\; i\in \mathbb{I}_{1},$ the differences $\tilde{u}^{(k)}(x) - \hat{u}^{(k)}(x),$ $k=0,1$ can be estimated from the Cauchy problem

  \begin{equation}\label{IVP_for_estimating_norm_of_Z}
    \dot{Z}_{i}(x) = \left[
                       \begin{array}{cc}
                         0 & 1 \\
                         \alpha_{i-1}(\tilde{u}^{\prime}(x_{i-1}), \tilde{u}(x_{i-1}), x) & 0 \\
                       \end{array}
                     \right]Z_{i}(x)
  \end{equation}
  $$+ \left[
                                         \begin{array}{c}
                                           0 \\
                                           \left(N(\hat{u}(x), x) - \alpha_{i-1}(\tilde{u}^{\prime}(x_{i-1}), \tilde{u}(x_{i-1}), x)\right)\hat{u}(x) \\
                                         \end{array}
                                       \right],$$
  $$x\in [x_{i-1}, x_{i}],\; Z_{i}(x_{i-1}) = Z_{i-1}(x_{i-1}),$$
  where
  $$Z_{i}(x) = \left[
                 \begin{array}{c}
                   z_{i}(x) \\
                   z^{\prime}_{i}(x) \\
                 \end{array}
               \right],\; z_{i}(x) = \hat{u}(x) - \tilde{u}(x),\; x\in[x_{i-1}, x_{i}],\;  Z_{0}(x)\equiv 0,\; i\in \mathbb{I}_{1}.
  $$

  From \eqref{IVP_for_estimating_norm_of_Z}, using \eqref{estimate_for_u_from_PL_theorem}, we get the recursive estimates
  \begin{equation}\label{recursive_estimates_for_Z_i}
    \|Z_{i}(x)\| \leq (1+h_{i}Q)\|Z_{i-1}\| + E\int\limits_{x_{i-1}}^{x_{i}}\|Z_{i}(\xi)\|d\xi + h_{i}^{3}K, \; x\in[x_{i-1}, x_{i}],\; i\in \mathbb{I}_{1},
  \end{equation}
  where
  $\|Z_{i}(x)\| = \max\{|z_{i}(x)|, |z^{\prime}_{i}(x)|\},\; x\in [x_{i-1}, x_{i}],\; \|Z_{i}\| \stackrel{def}{=} \max\limits_{x\in[x_{i-1}, x_{i}]}\|Z_{i}(x)\|, $ $h_{i} = x_{i}-x_{i-1}$
  $$E = \max\left\{1, R\right\},\; R = L_{0} + (c-a)L_{1}\left(\tilde{u}^{\prime\ast} + 1\right) \geq \max\limits_{x\in [x_{i-1}, x_{i}]}|\alpha_{i-1}(\tilde{u}^{\prime}(x_{i-1}), \tilde{u}(x_{i-1}), x)|,\; \forall i\in \mathbb{I}_{1},$$
  $$Q = \left(u_{b} + \varepsilon\right)\left(L_{1} + (c-a)\left(L_{1} + L_{2}\left(\tilde{u}^{\prime\ast} + 1\right)\right)\right) $$
  $$\geq \max\limits_{x\in[x_{i-1}, x_{i}]}|\hat{u}(x)(\alpha_{i-1}(\tilde{u}^{\prime}(x_{i-1}), \tilde{u}(x_{i-1}), x) - \alpha_{i-1}(\hat{u}^{\prime}(x_{i-1}), \hat{u}(x_{i-1}), x))|,\; \forall i\in \mathbb{I}_{1},$$
  $$K = \frac{u_{b}+\varepsilon}{2}\left(L_{2}\left(\tilde{u}^{\prime\ast} + \varepsilon + 1\right)^{2} + L_{1}L_{0}(u_{b}+\varepsilon)\right)\geq \frac{1}{2}\max\limits_{x\in [x_{i-1}, x_{i}]}|\hat{u}(x)\left(N(\hat{u}(x), x)^{\prime\prime}_{xx}\right)|,\; \forall i\in \mathbb{I}_{1},$$
  $$L_{k} = \max\left\{\left|\frac{d^{k}N(u,x)}{du^{i}dx^{j}}\right|\; : \; i,j \in \mathbb{N},\; i+j=k,\; u\in \mathbb{D}_{\varepsilon, u}, x\in[a,c]\right\}.$$

  In the estimates above we actively used result of Lemma \ref{lemma_about_u_tilde_prime_c_bounded_from_above}.

  Applying the Gronwall's inequality (see, for example, \cite[p. 42]{Teschl_ODE_and_DS}) to \eqref{recursive_estimates_for_Z_i} we get the estimate
\begin{equation}\label{estimate_due_to_Gronwall's inequality}
  \|Z_{i}\| \leq \left(\left(1+h_{i}Q\right)\|Z_{i-1}\| + h_{i}^{3}K\right)\exp\left(h_{i}E\right),\; i\in \mathbb{I}_{1},
\end{equation}
which, when applied recursively, yields the inequalities
\begin{equation}\label{h_squared_p_estimate}
  \|Z_{i}\| \leq K\sum\limits_{j=1}^{i}h_{j}^{3}\prod\limits_{k=j+1}^{i}\left(1+h_{k}Q\right)\prod\limits_{k=j}^{i}\exp\left(h_{k}E\right)
\end{equation}
$$ \leq K h^{2}\sum\limits_{j=1}^{i}h_{j}\prod\limits_{k=j+1}^{i}\left(1+h_{k}Q\right)\prod\limits_{k=j}^{i}\exp\left(h_{k}E\right)$$
$$\leq K h^{2} \sum\limits_{j=1}^{i}h_{j}\prod\limits_{k=j+1}^{i}\exp\left(h_{k}Q\right)\prod\limits_{k=j}^{i}\exp\left(h_{k}E\right)$$
$$\leq K h^{2} \sum\limits_{j=1}^{i}h_{j}\prod\limits_{k=1}^{i}\exp\left(h_{k}Q\right)\prod\limits_{k=1}^{i}\exp\left(h_{k}E\right) \leq K h^{2} \sum\limits_{j=1}^{i}h_{j}\exp\left((Q+E)\sum\limits_{k=1}^{i}h_{k}\right)$$
$$\leq K h^{2} (c-a)\exp\left((Q+E)(c-a)\right) \stackrel{def}{=} \kappa_{1}h^{2}.$$

So far we have proved that estimate \eqref{h_squared_p_estimate} holds true for all $i\in \mathbb{I}_{1}.$ At the same time, one can notice that the last expression in the chain of inequalities \eqref{h_squared_p_estimate} does not depend on $i.$ This, in particular, means that if we require that $h$ is small enough to ensure inequality
$$\kappa_{1} h^{2} \leq \frac{\varepsilon}{2},$$
then, using precisely the same reasoning as above, we can prove that solution $\mathbf{u}(x)$ of the Cauchy problem \eqref{theorem_about_approximation_aux_problem_vector_form} exists at least on
$$[a, c_{2}], \; c_{2} = \max\limits_{i\in \mathbb{I}_{1}}\left\{x_{i}\right\} + \frac{\varepsilon}{2M}.$$
Apparently, repeating the procedure not more than $N_{1}$ times, we will prove that solution $\hat{u}(x)$ exists on $[a,c]$ and estimate \eqref{approx_estimates_for_u_hat} holds true.

\begin{figure}[h!]
    \centering
    \begin{subfigure}[b]{0.48\textwidth}
    \begin{tikzpicture}
    \draw[<->] (6.5,0) node[below]{$x$} -- (0,0) --
    (0,4) node[left]{$u$};
    \draw (1, 1.65) node[left]{$\mathcal{A}$};
    \draw (5.3, 1.4) node[above]{$\mathcal{B}$};
    \draw (1, 0.6) node[left]{$\mathcal{C}$};
    \draw (3, 1.4) node[above]{$\mathcal{D}$};
    \draw (-0.1, 1.45) node[left]{$\tilde{u}(c)$} -- (6.1, 1.45);
    \draw (-0.1, 0.43) node[left]{$\tilde{u}(\hat{x}(\tilde{u}(c)))$} -- (4, 0.43);
    \draw (1, 0.1) -- (1, 2.5);
    \draw (3.17, -0.1) node[below]{$c$} -- (3.17, 2.5);
    \draw (1, 0.43) -- (6.1, 1.65);
    \draw[very thick] (0,0.25) to [out=8,in=-110] (5, 4.0) node[left]{$\tilde{u}(x)$};
    \draw[thick] (0,1.25) to [out=6,in=-110] (3.5, 4.0) node[left]{$\hat{u}(x)$};
    \end{tikzpicture}
    \caption{Case $\tilde{u}(c) < \hat{u}(c).$}\label{fig:M1a}
    \end{subfigure}
    \begin{subfigure}[b]{0.48\textwidth}
    \begin{tikzpicture}
    \draw[<->] (6.5,0) node[below]{$x$} -- (0,0) --
    (0,4) node[left]{$u$};
    \draw (1, 1.65) node[left]{$\mathcal{A}$};
    \draw (5.3, 1.4) node[above]{$\mathcal{B}$};
    \draw (1, 0.6) node[left]{$\mathcal{C}$};
    \draw (3, 1.4) node[above]{$\mathcal{D}$};
    \draw (-0.1, 1.45) node[left]{$\tilde{u}(c)$} -- (6.1, 1.45);
    \draw (-0.1, 0.43) node[left]{$\hat{u}(c)$} -- (4, 0.43);
    \draw (1, -0.1) node[below]{$c$} -- (1, 2.5);
    \draw (5.7, -0.1) node[below]{$c+\delta$} -- (5.7, 2.5);
    \draw (1, 0.43) -- (6.1, 1.65);
    \draw[thick] (0,0.25) to [out=8,in=-110] (5, 4.0) node[left]{$\hat{u}(x)$};
    \draw[very thick] (0,1.25) to [out=6,in=-110] (3.5, 4.0) node[left]{$\tilde{u}(x)$};
    \end{tikzpicture}
    \caption{Case $\tilde{u}(c) \geq \hat{u}(c).$}\label{fig:M1b}
    \end{subfigure}
    \caption{ Schematic illustration of the two possible mutual placements of the graphs of functions $\tilde{u}(x)$ and $\check{u}(x)$ near the point $x = c.$ } \label{fig:M1}
    \end{figure}
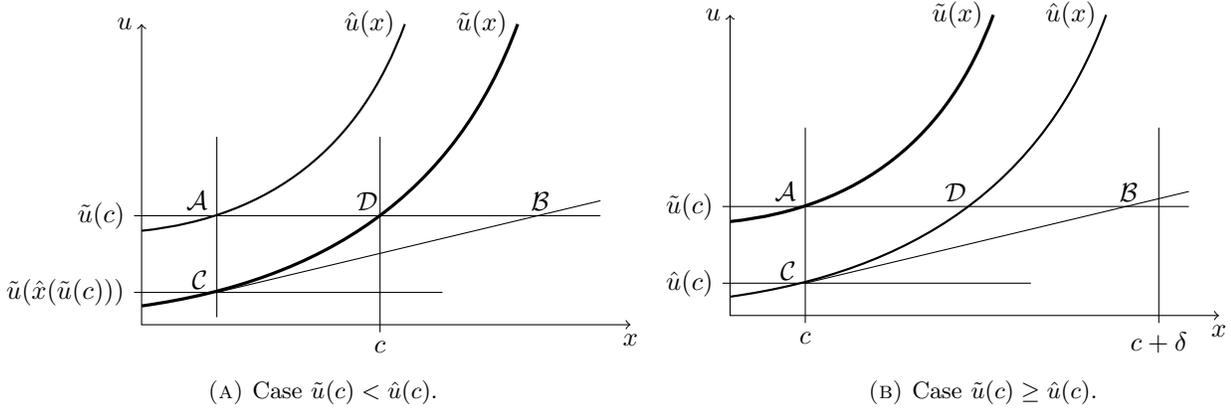

So far, we have proved that solution $\hat{u}(x)$ to IVP \eqref{Intro_Equation}, \eqref{theorem_about_approximation_init_conditions_for_aux_problem} exists at least on $[a,c].$ By means of inequality \eqref{h_squared_p_estimate}, it is not difficult to ensure that the solution, actually, exists on a bigger interval, namely,
\begin{equation}\label{definition_of_delta}
[a, c + \delta], \; \delta = \min\left\{\frac{\varepsilon}{2\bar{M}}, b-c\right\}, \; \bar{M} = \max\left\{\left\| \mathbf{F}(\mathbf{u}, x)\right\|\;:\; (\mathbf{u},x) \in \mathbb{D}_{\varepsilon, u} \times \mathbb{D}_{\varepsilon, u^{\prime}}\times [a,b]\right\}.
\end{equation}
As it was pointed out above, the latter fact yields us existence of $\hat{x}(u)$ on $[0, \tilde{u}(c)],$ provided that $h$ is sufficiently small.

{\it Part 2 : estimate \eqref{approx_estimates_for_x_hat} for $u = \tilde{u}(c)$. }

Now we want to proceed by proving estimate \eqref{approx_estimates_for_x_hat}. However, to do so, first we want to estimate expressions $|\tilde{x}(u) - \hat{x}(u)|,$ $|\tilde{x}^{\prime}(u) - \hat{x}^{\prime}(u)|$ at point $u=\tilde{u}(c).$ Let us begin by considering the case $\tilde{u}(c)<\hat{u}(c),$ which is illustrated on Fig. \ref{fig:M1a}. As it can be seen from the corresponding illustration, $|\tilde{x}(u) - \hat{x}(u)| = |\mathcal{AD}|,$ where segment $\mathcal{AD}$ is a part of cathetus $\mathcal{AB}$ of the right triangle $\triangle \mathcal{A}\mathcal{B}\mathcal{C}.$ The triangle is constructed in such a way, that its hypotenuse $\mathcal{B}\mathcal{C}$ lies on the tangent line to curve $u=\tilde{u}(x)$ at point $x = \hat{x}(\tilde{u}(c)),$ which yields us the estimate
\begin{equation}\label{inequality_x_tilde_minus_x_hat_case_a}
  |c - \hat{x}(\tilde{u}(c))| = |\mathcal{AD}| \leq |\mathcal{AB}| = \frac{|\mathcal{AC}|}{\tan \angle \mathcal{ABC}} \leq \frac{\kappa_{1} h^{2}}{\tan \angle \mathcal{ABC}} \leq \frac{\kappa_{1} h^{2}}{\tilde{u}^{\prime}_{\ast\ast}},
\end{equation}
where constant $\tilde{u}^{\prime}_{\ast\ast}$ is defined in \eqref{restriction_on_tilde_u_prime_from_below}.

Similarly, considering the case $\tilde{u}(c)\geq \hat{u}(c),$ which is illustrated on Fig. \ref{fig:M1b},
we get the estimate
\begin{equation}\label{inequality_x_tilde_minus_x_hat_case_b}
  |c - \hat{x}(\tilde{u}(c))| = |\mathcal{AD}| \leq |\mathcal{AB}| = \frac{|\mathcal{AC}|}{\tan \angle \mathcal{ABC}} \leq \frac{\kappa_{1} h^{2}}{\tan \angle \mathcal{ABC}} = \frac{\kappa_{1} h^{2}}{\hat{u}^{\prime}(c)} \leq \frac{\kappa_{1} h^{2}}{\tilde{u}^{\prime}_{\ast} - \kappa_{1} h^{2}},
\end{equation}
where in the last inequality we used estimate \eqref{h_squared_p_estimate} and Remark \ref{remark_about_u_tilde_prime_c_bounded from below}, i.e,
$$0 < \hat{u}^{\prime}(c) \geq \tilde{u}^{\prime}(c) - \kappa_{1} h^{2} \geq \tilde{u}^{\prime}_{\ast} - \kappa_{1} h^{2}.$$
Requiring that
\begin{equation}\label{requirement_x_tilde_minus_x_hat_leq_case_b}
\kappa_{1} h^{2} \leq \frac{\tilde{u}_{\ast}^{\prime}}{4}, \; |c - \hat{x}(\tilde{u}(c))| \leq \delta\footnote{See \eqref{definition_of_delta}.}
\end{equation}
and combining \eqref{inequality_x_tilde_minus_x_hat_case_b} with \eqref{inequality_x_tilde_minus_x_hat_case_a} we get the estimate
\begin{equation}\label{estimate_x_tilde_minus_x_hat_case_a}
  |c - \hat{x}(\tilde{u}(c))| \leq \kappa_{1}h^{2}\max\left\{\frac{4}{3\tilde{u}^{\prime}_{\ast}}, \frac{1}{\tilde{u}^{\prime}_{\ast\ast}}\right\}.
\end{equation}
Requirement \eqref{requirement_x_tilde_minus_x_hat_leq_case_b} can be restated in terms of another restriction on $h$ as follows:
\begin{equation}\label{bound_of_magnitude_of_h_to_fulfill_requirement_x_tilde_minus_x_hat_leq_case_b}
  h^{2} \leq \frac{\tilde{u}^{\prime}_{\ast}}{4\kappa_{1}}\min\left\{1, 3\delta\right\}.
\end{equation}

By requiring additionally
$$h^{2}\leq \frac{\tilde{u}^{\prime}_{\ast}}{4\kappa_{1}\tilde{u}^{\prime\prime\ast}}\min\left\{\frac{3\tilde{u}^{\prime}_{\ast}}{4}, \tilde{u}^{\prime}_{\ast\ast}\right\},$$
which in the light of estimate \eqref{estimate_x_tilde_minus_x_hat_case_a} yields us the inequality
\begin{equation}\label{requirement_x_tilde_minus_x_hat_leq}
  |c - \hat{x}(\tilde{u}(c))|\tilde{u}^{\prime\prime\ast} \leq \frac{\tilde{u}^{\prime}_{\ast}}{4},
\end{equation}
where
$$\tilde{u}^{\prime\prime\ast} = \left(L_{0} + (c-a)L_{1}\left(\tilde{u}^{\prime\ast} + 1\right)\right)u_{b} \geq \tilde{u}^{\prime\prime}(x) = \alpha(\mathbb{P}_{x}(\tilde{u}^{\prime}(x)), \mathbb{P}_{x}(\tilde{u}(x)), x)\tilde{u}(x),\; \forall x\in [a,c],$$
we ensure that
\begin{equation}\label{estimate_for_u_hat_prime}
  \hat{u}^{\prime}(\hat{x}(\tilde{u}(c))) = \frac{1}{\hat{x}^{\prime}(\tilde{u}(c))} \geq \frac{\tilde{u}^{\prime}_{\ast}}{2}.
\end{equation}
Indeed:
$$\hat{u}^{\prime}(\hat{x}(\tilde{u}(c))) - \frac{\tilde{u}^{\prime}_{\ast}}{2} = \hat{u}^{\prime}(\hat{x}(\tilde{u}(c))) - \tilde{u}^{\prime}(\hat{x}(\tilde{u}(c))) + \tilde{u}^{\prime}(\hat{x}(\tilde{u}(c))) - \tilde{u}^{\prime}(c) + \tilde{u}^{\prime}(c) - \frac{\tilde{u}^{\prime}_{\ast}}{2}$$
$$\geq \frac{\tilde{u}^{\prime}_{\ast}}{2} - |\hat{u}^{\prime}(\hat{x}(\tilde{u}(c))) - \tilde{u}^{\prime}(\hat{x}(\tilde{u}(c)))| - |\tilde{u}^{\prime}(\hat{x}(\tilde{u}(c))) - \tilde{u}^{\prime}(c)|$$
$$\geq \frac{\tilde{u}^{\prime}_{\ast}}{2} - \kappa_{1}h^{2} - |c - \hat{x}(\tilde{u}(c))|\tilde{u}^{\prime\prime\ast} \geq \frac{\tilde{u}^{\prime}_{\ast}}{2} - \frac{\tilde{u}^{\prime}_{\ast}}{4} - \frac{\tilde{u}^{\prime}_{\ast}}{4} = 0.$$

As for the corresponding estimate for the derivatives, i.e.  $|\tilde{x}^{\prime}(\tilde{u}(c)) - \hat{x}^{\prime}(\tilde{u}(c))|,$ we can obtain it in a unified way without a need to separately consider the two cases introduced above (see Fig. \ref{fig:M1a}, \ref{fig:M1b}):
\begin{equation}\label{estimate_for_derivative_total}
  |\tilde{x}^{\prime}(\tilde{u}(c)) - \hat{x}^{\prime}(\tilde{u}(c))| = \left|\frac{\tilde{u}^{\prime}(c) - \hat{u}^{\prime}(\hat{x}(\tilde{u}(c)))}{\tilde{u}^{\prime}(c)\hat{u}^{\prime}(\hat{x}(\tilde{u}(c)))}\right|\leq \frac{2}{\left(\tilde{u}^{\prime}_{\ast}\right)^{2}}\left|\tilde{u}^{\prime}(c) - \hat{u}^{\prime}(\hat{x}(\tilde{u}(c)))\right|\footnote{Here we have used inequality \eqref{estimate_for_u_hat_prime}.}
\end{equation}
$$\leq \frac{2}{\left(\tilde{u}^{\prime}_{\ast}\right)^{2}}\left(\left|\tilde{u}^{\prime}(c) - \hat{u}^{\prime}(c)\right| + \left|\hat{u}^{\prime}(c) - \hat{u}^{\prime}(\hat{x}(\tilde{u}(c)))\right|\right)\leq \frac{2h^{2}\kappa_{1}}{\left(\tilde{u}^{\prime}_{\ast}\right)^{2}}\left(1 + \hat{u}^{\prime\prime\ast}\max\left\{\frac{4}{3\tilde{u}^{\prime}_{\ast}}, \frac{1}{\tilde{u}^{\prime}_{\ast\ast}}\right\}\right),$$
where
$$\hat{u}^{\prime\prime\ast} = \max\left\{|\mathcal{N}(u,x)| : \; u \in [0, u_{b}],\; x\in [a, c+\delta]\right\} \geq \hat{u}^{\prime\prime}(x),\; \forall x\in [a, c+\delta].$$

{\it Part 3 : existence of $\hat{x}(u)$ on $[\tilde{u}(c), u_{b}]$ and estimate \eqref{approx_estimates_for_x_hat}. }

Let us assume that $h$ is small enough to ensure inequalities (see estimates \eqref{estimate_x_tilde_minus_x_hat_case_a}, \eqref{estimate_for_derivative_total})
\begin{equation}\label{basic_assumption_part_3}
|c - \hat{x}(\tilde{u}(c))|, |\tilde{x}^{\prime}(\tilde{u}(c)) - \hat{x}^{\prime}(\tilde{u}(c))| \leq \frac{\varepsilon}{2}.
\end{equation}
Then, according to the {\it Picard-Lindelof Theorem} (see, for example, \cite[p. 350]{Kelley_Peterson_2010}), function $\hat{x}(u)$ exists at least on
\begin{equation}\label{initial_existance_interval_for_inverse_function}
  [\tilde{u}(c), d_{s}],
\end{equation}
for $s = 1,$ where
$$d_{i} = \min\left\{u_{b}, d_{i-1}+\frac{\varepsilon}{2M}\right\},\; d_{0} = \tilde{u}(c),$$
$$M = \max\left\{\left\|F(x^{\prime}, x, u)\right\|\; : x^{\prime} \in \mathbb{D}_{\varepsilon, x^{\prime}},\; x\in \mathbb{D}_{\varepsilon, x},\; u \in \mathbb{D}_{\varepsilon, u}\right\}$$
$$ F(x^{\prime}, x, u) = \left[
                                                     \begin{array}{c}
                                                       -\mathcal{N}(u, x)\left(x^{\prime}\right)^{3} \\
                                                       x^{\prime} \\
                                                     \end{array}
                                                   \right],\; \mathbb{D}_{\varepsilon, x^{\prime}} = [0, 1/\tilde{u}^{\prime}_{\ast} + \varepsilon],\; \mathbb{D}_{\varepsilon, x} = [c-\varepsilon, b+\varepsilon],\; \mathbb{D}_{\varepsilon, u} = [- \varepsilon, u_{b} + \varepsilon]
$$
and constant $\tilde{u}^{\prime}_{\ast}$ is defined in \eqref{estimate_from_above_for_u_peimw_c}.

For the sake of simplicity, we assume that
$$h < d_{s} - d_{s-1},\; \forall s : d_{s} \neq d_{s-1},$$
which guarantees that a set of indices $$\mathbb{J}_{s} = \left\{i\in \mathbb{N}\; |\; d_{s-1} < \bar{u}_{i} \leq d_{s} \right\}$$ is non-empty as long as $d_{s} \neq d_{s-1}.$

It is not difficult to verify that on interval \eqref{initial_existance_interval_for_inverse_function} functions $\hat{x}(u)$ and $\tilde{x}(u)$ satisfy equalities
\begin{equation}\label{x_hat_expression}
  \hat{x}^{\prime}(u) = \frac{1}{\sqrt{(\hat{x}^{\prime}(\tilde{u}(c)))^{-2} + 2 \int\limits_{\tilde{u}(c)}^{u}\mathcal{N}(\eta, \hat{x}(\eta))d\eta}},
\end{equation}
and
\begin{equation}\label{x_tilde_expression}
  \tilde{x}^{\prime}(u) = \frac{1}{\sqrt{(\tilde{x}^{\prime}(\tilde{u}(c)))^{-2} - 2\int\limits_{\tilde{u}(c)}^{u}\beta(\mathbb{P}_{u}(\tilde{x}^{\prime}(\eta)), \mathbb{P}_{u}(\tilde{x}(\eta)), \eta)d\eta}}
\end{equation}
respectively.

Subtracting \eqref{x_tilde_expression} from \eqref{x_hat_expression} we get

$$\hat{x}^{\prime}(u) - \tilde{x}^{\prime}(u) = \left(\hat{x}^{\prime}(\tilde{u}(c)) - \tilde{x}^{\prime}(\tilde{u}(c))\right)\frac{\hat{x}^{\prime}(\tilde{u}(c)) + \tilde{x}^{\prime}(\tilde{u}(c))}{\left(\hat{x}^{\prime}(\tilde{u}(c))\tilde{x}^{\prime}(\tilde{u}(c))\right)^{2}}\frac{\left(\hat{x}^{\prime}(u)\tilde{x}^{\prime}(u)\right)^{2}}{\hat{x}^{\prime}(u) + \tilde{x}^{\prime}(u)} $$
$$+ \left(2 \int\limits_{\tilde{u}(c)}^{u}\mathcal{N}(\eta, \hat{x}(\eta))d\eta + 2\int\limits_{\tilde{u}(c)}^{u}\beta(\mathbb{P}_{u}(\tilde{x}^{\prime}(\eta)), \mathbb{P}_{u}(\tilde{x}(\eta)), \eta)d\eta\right)\frac{\left(\hat{x}^{\prime}(u)\tilde{x}^{\prime}(u)\right)^{2}}{\hat{x}^{\prime}(u) + \tilde{x}^{\prime}(u)},$$
which, together with the obvious inequality
$$0 < \frac{\hat{x}^{\prime}(\tilde{u}(c)) + \tilde{x}^{\prime}(\tilde{u}(c))}{\left(\hat{x}^{\prime}(\tilde{u}(c))\tilde{x}^{\prime}(\tilde{u}(c))\right)^{2}}\frac{\left(\hat{x}^{\prime}(u)\tilde{x}^{\prime}(u)\right)^{2}}{\hat{x}^{\prime}(u) + \tilde{x}^{\prime}(u)} \leq 1,$$
allows us to proceed with the estimates
\begin{equation}\label{estimate_for_u_hat_prime_minus_x_tilde_prime}
  \|\hat{x}^{\prime}(u) - \tilde{x}^{\prime}(u)\|_{n,0} \leq \|\hat{x}^{\prime}(u) - \tilde{x}^{\prime}(u)\|_{-1,1}(1+u_{b}E) + Kh^{2}
\end{equation}
$$+ E\sum\limits_{i=0}^{n-1}\bar{h}_{i+1}\|\hat{x}(u) - \tilde{x}(u)\|_{i,1},\; \forall n\in \mathbb{J}_{s},$$
\begin{equation}\label{estimate_hat_x_minus_tilde_x}
  \|\hat{x}(u) - \tilde{x}(u)\|_{n,0} \leq |\hat{x}(\tilde{u}(c)) - \tilde{x}(\tilde{u}(c))| + \sum\limits_{j=0}^{n}\bar{h}_{j}\|\hat{x}^{\prime}(u) - \tilde{x}^{\prime}(u)\|_{j, 0}
\end{equation}
$$\leq \|\hat{x}(u) - \tilde{x}(u)\|_{-1,1}\left(1+u_{b}(1+u_{b}E)\right) + u_{b}K h^{2} + E\sum\limits_{j=0}^{n}\bar{h}_{j}\sum\limits_{i=0}^{j-1}\bar{h}_{i+1}\|\hat{x}(u) - \tilde{x}(u)\|_{i,1} $$
$$\leq \|\hat{x}(u) - \tilde{x}(u)\|_{-1,1}\left(1+u_{b}(1+u_{b}E)\right) + u_{b}K h^{2} + u_{b}E\sum\limits_{i=0}^{n-1}\bar{h}_{i+1}\|\hat{x}(u) - \tilde{x}(u)\|_{i,1},\; \forall n\in \mathbb{J}_{s},$$
where
$$\bar{h}_{i} = \bar{u}_{i+1} - \bar{u}_{i},\; \|f(u)\|_{i,k} \stackrel{def}{=} \max\left\{|f^{(m)}(u)|\; : \; u\in [\bar{u}_{i}, \bar{u}_{i+1}],\, m\in\overline{0\ldots k}\right\},\; i \geq 0, $$
$$\|f(u)\|_{-1,k} \stackrel{def}{=} \max\left\{|f^{(m)}(\tilde{u}(c))|,\;m\in\overline{0\ldots k}\right\},$$
$$E = \frac{8}{3\left(\tilde{u}^{\prime}_{\ast}\right)^{3}}\left(L_{1} + \left(L_{2}\left(1+ 1/\tilde{u}^{\prime}_{\ast}\right) + L_{1}\right)u_{b}\right)\footnote{Here have we used an estimate $\frac{\left(\hat{x}^{\prime}(u)\tilde{x}^{\prime}(u)\right)^{2}}{\hat{x}^{\prime}(u) + \tilde{x}^{\prime}(u)} = \left(\hat{x}^{\prime}(u)\tilde{x}^{\prime}(u)\right)/\left(\frac{1}{\hat{x}^{\prime}(u)} + \frac{1}{\tilde{x}^{\prime}(u)}\right) \leq \frac{4}{3\left(\tilde{u}^{\prime}_{\ast}\right)^{3}},\; \forall u\geq \tilde{u}(c),$ which follows from inequality \eqref{estimate_for_u_hat_prime} and the fact that functions $\hat{x}(u),$ $\tilde{x}(u)$ are concave (see reresentations \eqref{x_hat_expression}, \eqref{x_tilde_expression} in the light of conditions \eqref{lema_1_positiveness_of_nonlinearity}, \eqref{positiveness_of_N_derivative_with_respect_to_x}).}$$
$$\geq 2\frac{\left(\hat{x}^{\prime}(u)\tilde{x}^{\prime}(u)\right)^{2}}{\hat{x}^{\prime}(u) + \tilde{x}^{\prime}(u)} |\beta(\mathbb{P}_{u}(\tilde{x}^{\prime}(u)), \mathbb{P}_{u}(\tilde{x}(u)), u) - \beta(\mathbb{P}_{u}(\hat{x}^{\prime}(u)), \mathbb{P}_{u}(\hat{x}(u)), u)|,\; \forall u\in [\tilde{u}(c), u_{b}],$$

$$K = \frac{4}{3\left(\tilde{u}^{\prime}_{\ast}\right)^{3}}\left(L_{2}\left(1+\frac{2}{\tilde{u}^{\prime}_{\ast}}\right)^{2} + L_{0}L_{1}\frac{8}{\left(\tilde{u}^{\prime}_{\ast}\right)^{3}}\right)u_{b} $$
$$\geq \frac{\left(\hat{x}^{\prime}(u)\tilde{x}^{\prime}(u)\right)^{2}}{\hat{x}^{\prime}(u) + \tilde{x}^{\prime}(u)}\left|\frac{d^{2}\mathcal{N}(u, \hat{x}(u))}{du^{2}}\right|u_{b},\; \forall u\in[\tilde{u}(c), u_{b}].$$
$$L_{k} = \max\left\{\left|\frac{d^{k}\mathcal{N}(u, x)}{du^{i}dx^{j}}\right|\; : \; i,j, \in \mathbb{N},\; i+j = k,\; u\in \mathbb{D}_{\varepsilon, u}, x\in \mathbb{D}_{\varepsilon, x}\right\}.$$

Let us consider an auxiliary sequence $\{\mu_{i}\}$ defined in the following way (see estimates \eqref{estimate_x_tilde_minus_x_hat_case_a}, \eqref{estimate_for_derivative_total})
$$\mu_{0} = Q_{1}\|\hat{x}(u) - \tilde{x}(u)\|_{-1,1} + Q_{2}K h^{2} \leq h^{2}\mu(\kappa_{1}),$$
$$\mu(\kappa_{1}) \stackrel{def}{=} Q_{1}\kappa_{1}\max\left\{\max\left\{\frac{4}{3\tilde{u}^{\prime}_{\ast}}, \frac{1}{\tilde{u}^{\prime}_{\ast\ast}}\right\}, \frac{2}{\left(\tilde{u}^{\prime}_{\ast}\right)^{2}}\left(1 + \hat{u}^{\prime\prime\ast}\max\left\{\frac{4}{3\tilde{u}^{\prime}_{\ast}}, \frac{1}{\tilde{u}^{\prime}_{\ast\ast}}\right\}\right)\right\} + Q_{2}K,$$
$$ Q_{1} = \max\{1+u_{b}E, 1+ u_{b}(1+u_{b}E)\},\; Q_{2} = \max\{1, u_{b}\},$$
$$\mu_{i} = (1+ Q_{2}Eh_{i})\mu_{i-1} = (1+ Q_{2}Eh_{i})^{i}\mu_{0} \leq h^{2}\kappa_{2},\; \forall i\in \mathbb{J}_{s},$$
where
\begin{equation}\label{kappa_2_definition}
  \kappa_{2} = \boldsymbol{\kappa}_{2}(\kappa_{1})\footnote{Here we want to emphasize the fact that $\kappa_{2}$ is a function of $\kappa_{1},$ i.e. estimate \eqref{approx_estimates_for_x_hat} depends on estimate \eqref{approx_estimates_for_u_hat}. This will be used later in the proof of Theorem \ref{Main_theorem_abour_apptoximation_prorerties_of_the_SI_method}. } = \exp\left(u_{b}Q_{2}E\right)\mu(\kappa_{1}).
\end{equation}
Comparing the definition of $\mu_{i}$ with estimates \eqref{estimate_for_u_hat_prime_minus_x_tilde_prime} and \eqref{estimate_hat_x_minus_tilde_x}, one can conclude that
\begin{equation}\label{estimate_for_x_hat_minus_x_tilde_norm_one}
  \|\hat{x}(u) - \tilde{x}(u)\|_{i, 1} \leq \mu_{i} \leq h^{2}\kappa_{2},\; \forall i\in \mathbb{J}_{s}.
\end{equation}
If $d_{1} = u_{b}$ then the proof is complete. Otherwise, requiring $h$ to be small enough to ensure inequality $$h^{2}\kappa_{2} \leq \frac{\varepsilon}{2},$$ and using the {\it Picard-Lindelof Theorem} again, we conclude that the solution $\hat{x}(u)$ exists at least on interval \eqref{initial_existance_interval_for_inverse_function} for $s=2,$ and, literally repeating all the reasoning above, we again come to estimate \eqref{estimate_for_x_hat_minus_x_tilde_norm_one} for $s=2$. Apparently, after a finite number of iterations we will achieve the equality $d_{s} = u_{b},$ which, apparently, ensures the existence of $\hat{x}(u)$ on $[\tilde{u}(c), u_{b}]$ as well as estimate \eqref{approx_estimates_for_x_hat}.
\end{proof}

Applying a technique similar to the one used in the proof above, one can prove a "symmetric" version of Theorem \ref{aux_theorem_about_approx_properties} stated below.
\begin{theorem}\label{aux_theorem_about_approx_properties_1}
Let the assumptions of Theorem \ref{aux_theorem_about_approx_properties} hold true. Then, for $h$ sufficiently small, there exists a function $\check{u}(x) = \check{u}(x, h) \in C^{2}([a,b]),$ which satisfies equation \eqref{Intro_Equation} subjected to initial conditions
\begin{equation}\label{init_conditions_for_u_check}
  \check{u}(b) = \tilde{x}^{-1}(b) = u_{b},\; \check{u}^{\prime}(b) = \frac{1}{\tilde{x}^{\prime}(u_{b}, h)}.
\end{equation}
and the following estimates hold true:
\begin{equation}\label{approx_estimates_for_u_check}
  \|\tilde{u}(x) - \check{u}(x)\|_{ [a, c], 1} = h^{2}\kappa_{3}, \; \|\tilde{x}(u) - \check{x}(u)\|_{[\tilde{u}(c), u_{b}], 1} = h^{2}\kappa_{4},
\end{equation}
where the constants $\kappa_{3}, \kappa_{4} > 0$ depend on BVP \eqref{Intro_Equation}, \eqref{Intro_boundary_conditions} only, $\check{x}(u) \stackrel{def}{=} \check{u}^{-1}(u).$
\end{theorem}

 Now we are in a position to prove a theorem about approximation properties of the SI-method with respect to the solution of BVP \eqref{Intro_Equation}, \eqref{Intro_boundary_conditions}. A similar statement was formulated in \cite{Makarov_Dragunov_2019} (see Propositions 1 and 2) without a proof.

\begin{theorem}\label{Main_theorem_abour_apptoximation_prorerties_of_the_SI_method}
Let condition \eqref{Intro_Nonlin_cond} as well as the assumptions of Theorem \ref{aux_theorem_about_approx_properties} hold true. Then, for $h$ \eqref{definition_of_h} sufficiently small, the following estimates hold true:
\begin{equation}\label{asymptotic_estimate_u_minus_u_tilde}
  \|u(x) - \tilde{u}(x)\|_{[a,c],1} \leq \kappa_{S}h^{2},
\end{equation}
\begin{equation}\label{asymptotic_estimate_x_minus_x_tilde}
  \|x(u) - \tilde{x}(u)\|_{[\tilde{u}(c), u_{b}], 1} \leq \kappa_{I}h^{2},
\end{equation}
where $u(x)$ is the solution to BVP \eqref{Intro_Equation}, \eqref{Intro_boundary_conditions}, $x(\cdot) = u^{-1}(\cdot)$ and constants $\kappa_{S}, \kappa_{I}$ depend on BVP \eqref{Intro_Equation}, \eqref{Intro_boundary_conditions} only.
\begin{proof}
It is easy to see, that under the assumptions of the theorem, the results of Theorems \ref{aux_theorem_about_approx_properties}, \ref{aux_theorem_about_approx_properties_1} are also valid.

If we regard solution $u(x)$ as a function of the boundary condition at point $x = a,$ i.e., $u(x) = u(x, u(a)),$ then, by the definition of function $\check{u}(x),$ introduced in Theorem \ref{aux_theorem_about_approx_properties_1}, we have that
 $$\check{u}(x) = u(x, \check{u}(a)).$$
 From Theorem \ref{aux_theorem_about_approx_properties_1} and the {\it Theorem about differentiability of solutions of BVPs with respect to boundary conditions} (see \cite[Theorem 1]{Vidossich_Giovanni_2001_differentiability}) it follows that (provided that $h$ is sufficiently small)
 \begin{equation}\label{estimate_check_u_minus_u}
   \|\check{u}(x) - u(x)\|_{[a,c],1} = \|u(x, \check{u}(a)) - u(x,0)\|_{[a,c],1}\leq \rho_{\varepsilon}\kappa_{3}h^{2},
 \end{equation}
 where
 $$\rho_{\varepsilon} = \max\left\{|u^{\prime}_{r}(x, r)|, |u^{\prime\prime}_{rx}(x, r)|\; :\; x\in[a,c],\;r\in [-\varepsilon, \varepsilon] \right\}, \varepsilon = \varepsilon(h) = \kappa_{3}h^{2}.$$

Let us get an estimate from above for the value of $\rho_{\varepsilon}.$

According to Theorem 1 from \cite{Vidossich_Giovanni_2001_differentiability}, function $u^{\prime}_{r}(x, r)$ is the solution to the boundary value problem
\begin{equation}\label{boundary_value_problem_for_derivarive_with_respect_to_r}
  v^{\prime\prime}(x) = \mathcal{N}^{\prime}_{u}(u(x, r),x)v(x),\; v(a) = 1,\; v(b) = 0.
\end{equation}
Condition \eqref{Intro_Nonlin_cond} guarantees that $\mathcal{N}^{\prime}_{u}(u(x, r),x) \geq 0,$ which allows us to apply the {\it maximum principle} (see, for example, Theorem 3 from \cite[p. 6]{Protter_Weinberger_Max_principle}) to the solution of problem \eqref{boundary_value_problem_for_derivarive_with_respect_to_r} and conclude that $u^{\prime}_{r}(x, r),$ as a function of $x,$ is decreasing on $[a, b]$ and thus
\begin{equation}\label{estimate_for_derivarive_with_respect_to_r}
  0\leq u^{\prime}_{r}(x, r) \leq 1, \; \forall x \in [a,b],\; \forall r\in \mathbb{R}.
\end{equation}
Integrating both sides of equation \eqref{boundary_value_problem_for_derivarive_with_respect_to_r} twice with respect to $x$ (with $v(x)= u^{\prime}_{r}(x,r)$), we get the inequality
$$u^{\prime}_{r}(x,r) = 1 + u^{\prime\prime}_{rx}(a,r)(x-a) + \int\limits_{a}^{x}\int\limits_{a}^{\eta}\mathcal{N}^{\prime}_{u}(u(\xi,r), \xi)d\xi d\eta \geq 0,\; \forall x\in[a,c], $$
which allows us to estimate $u^{\prime\prime}_{rx}(x,r) \leq 0$ from below as follows
\begin{equation}\label{estimate_for_derivarive_with_respect_to_r_and_x}
  0 \geq u^{\prime\prime}_{rx}(x,r) \geq u^{\prime\prime}_{rx}(a,r) \geq -\frac{1 + \int\limits_{a}^{x}\int\limits_{a}^{\eta}\mathcal{N}^{\prime}_{u}(u(\xi,r), \xi)d\xi d\eta}{c-a}.
\end{equation}
From \eqref{estimate_for_derivarive_with_respect_to_r_and_x}, using inequalities \eqref{estimate_for_derivarive_with_respect_to_r}, we get
\begin{equation}\label{estimate_for_derivarive_with_respect_to_r_and_x_1}
\max\limits_{\overset{x\in[a,c]}{r\in [-\varepsilon, \varepsilon]}}|u^{\prime\prime}_{rx}(x, r)| \leq \frac{1}{c-a} + (c-a)\max\limits_{\overset{x\in[a,c]}{u\in [-\varepsilon, u_{b} + \varepsilon]}}|\mathcal{N}^{\prime}_{u}(u,x)|,
\end{equation}
since
$$\max\limits_{r\in[-\varepsilon, \varepsilon]}|u(x,r) - u(x)| = \max\limits_{r\in[-\varepsilon, \varepsilon]}|u(x,r) - u(x, 0)| \leq \varepsilon \max\limits_{r\in[-\varepsilon, \varepsilon]}|u^{\prime}_{r}(x,r)| \leq \varepsilon,\; \forall x\in[a, c]$$
and $0 \leq u(x) \leq u_{b}.$

Finally, from inequalities \eqref{estimate_for_derivarive_with_respect_to_r} and \eqref{estimate_for_derivarive_with_respect_to_r_and_x_1}, we get the estimate
$$\rho_{\varepsilon} \leq \max\left\{1, \frac{1}{c-a} + (c-a)\max\limits_{\overset{x\in[a,c]}{u\in [-\varepsilon, u_{b} + \varepsilon]}}|\mathcal{N}^{\prime}_{u}(u,x)|\right\}.$$

Combining the result of Theorem \ref{aux_theorem_about_approx_properties_1} with estimate \eqref{estimate_check_u_minus_u} we get
 \begin{equation}\label{estimate_u_tilde_minus_u}
   \|\tilde{u}(x) - u(x)\|_{[a,c],1} \leq \|\tilde{u}(x) - \check{u}(x)\|_{[a,c],1} + \|\check{u}(x) - u(x)\|_{[a,c],1} \leq \kappa_{3}h^{2}(1+\rho_{\varepsilon}),
 \end{equation}
 which yields inequality \eqref{asymptotic_estimate_u_minus_u_tilde} with $\kappa_{S} = \kappa_{3}(1+\rho_{\varepsilon}).$

 Now with inequality \eqref{estimate_u_tilde_minus_u} at our hands, we can literally repeat all the reasoning done in parts 2 and 3 of the proof of Theorem
 \ref{aux_theorem_about_approx_properties} (tightening restriction on $h,$ if required) and get the estimate (see \eqref{estimate_for_x_hat_minus_x_tilde_norm_one})
 $$\|\tilde{x}(u) - x(u)\|_{[\tilde{u}(c),u_{b}],1} \leq h^{2}\boldsymbol{\kappa}_{2}(\kappa_{3}(1+\rho_{\varepsilon})),$$
 where function $\boldsymbol{\kappa}_{2}(\cdot)$ is defined in \eqref{kappa_2_definition}, which ensures inequality \eqref{asymptotic_estimate_x_minus_x_tilde} with $\kappa_{I} = \boldsymbol{\kappa}_{2}(\kappa_{3}(1+\rho_{\varepsilon})).$

 This completes the proof.
\end{proof}

\end{theorem}

\section{Implementation aspects of the SI-method}\label{section_implementation_aspects}
In the current section we would like to discuss some technical details of the SI-method's implementation which is freely available at the public repository \linebreak \url{https://github.com/imathsoft/MathSoftDevelopment}. What follows is not the only possible way how the SI-method can be implemented in practice but rather an attempt to share our experience in that area by giving some guide lines.

\subsection{Step functions}

   To describe the SI-method's implementation we need to introduce a concept of { \it step functions.} In the current paper we define the step functions in a slightly different way as compared to how they were defined in \cite{Makarov_Dragunov_2019} while still keeping the same notation. The new definition better fits into the theoretical framework presented in the current paper.

    Throughout this section we will refer to $U(x) = U(A, B, C, D, x)$ as the {\it straight step function} and define it to be the solution to IVP
    \begin{equation}\label{SI_method_definition_of_U}
      U^{\prime\prime}(s) = \left(As + B\right)U(s), \;\; U(0) = D,\; U^{\prime}(0) = C,
    \end{equation}
    whereas function $V(s) = V(\bar{A}, \bar{B}, \bar{C}, \bar{D}, s),$ satisfying the nonlinear IVP
    \begin{equation}\label{SI_method_definition_of_V}
      V^{\prime\prime}(s) = \left(\bar{A}s + \bar{B}\right)\left(V^{\prime}(s)\right)^{3}, \;\; V(0) = \bar{D}, \; V^{\prime}(0) = \bar{C},
    \end{equation}
    will be referred to as the {\it inverse step function}.
    It is easy to see that functions $\tilde{u}(x)$ and $\tilde{x}(u),$ satisfying equations \eqref{discretized_equation_straight} and \eqref{discretized_equation_inverse} respectively, can be expressed through the step functions in the following way
    \begin{equation}\label{u_tilde_through_step_function}
      \tilde{u}(x) = U\left(N^{\prime}_{u}(\tilde{u}(x_{i}), x_{i})\tilde{u}^{\prime}(x_{i}) + N^{\prime}_{x}(\tilde{u}(x_{i}), x_{i}), N(\tilde{u}(x_{i}), x_{i}), \tilde{u}^{\prime}(x_{i}), \tilde{u}(x_{i}), x-x_{i}\right),
    \end{equation}
    $$x\in [x_{i}, x_{i+1}], \; i\in \overline{0, N_{1}-1},$$
    \begin{equation}\label{x_tilde_through_step_function}
      \tilde{x}(u) = V\left(-\mathcal{N}^{\prime}_{u}(\bar{u}_{i}, \tilde{x}(\bar{u}_{i})) - \mathcal{N}^{\prime}_{x}(\bar{u}_{i}, \tilde{x}(\bar{u}_{i}))\tilde{x}^{\prime}(\bar{u}_{i}), - \mathcal{N}(\bar{u}_{i}, \tilde{x}(\bar{u}_{i})), \tilde{x}^{\prime}(\bar{u}_{i}), \tilde{x}(\bar{u}_{i}), u - \bar{u}_{i}\right),
    \end{equation}
    $$u\in [\bar{u}_{i}, \bar{u}_{i+1}],\; i\in \overline{0, N_{2}-1}.$$
    Notice, that equalities \eqref{u_tilde_through_step_function}, \eqref{x_tilde_through_step_function} require functions $U(s)$ and $V(s)$ to be approximated for rather small values of their arguments, i.e. $0 \leq s < h.$ Such approximations can be efficiently constructed via the Tailor series expansions (see, for example, \cite{Hairer_Wanner_Non_Stiff}).

\subsection{System of nonlinear equations}
Equalities \eqref{u_tilde_through_step_function}, \eqref{x_tilde_through_step_function} allow us to reduce the system of differential equations with boundary and matching conditions \eqref{discretized_equation_straight} -- \eqref{matching_conditions_u} to a system of nonlinear equations with respect to unknown values $\tilde{u}(x_{i}),$ $\tilde{x}(\bar{u}_{i}).$ The latter system can be solved by some iteration technique, e.g. the Newton's method (see, for example, \cite[Section 2.3]{Ascher_1988}). In our implementation, to approximate partial derivatives of the step functions with respect to parameters $A,\ldots, D,$ $\bar{A},\ldots, \bar{D},$ which are required by the Newton's method, we use the method of {\it algorithmic differentiation} (AD) (see, for example, \cite{Griewank_algorithmic_differentiation}). The AD is easy to implement, it provides enough flexibility for possible experiments (one can re-define the step functions without caring too much about the evaluation of their derivatives) and shows quite good performance, especially if the step functions are evaluated through the Tailor series expansions.

The general approach for building and solving the nonlinear system with respect to values $\tilde{u}(x_{i}),$ $\tilde{x}(\bar{u}_{i})$ (when solution $u(x)$ is not necessary monotone and convex) is quite thoroughly described in \cite[Section 3]{Makarov_Dragunov_2019}. Notice that application of an iteration technique for solving the system could cause a "mesh drifting" near the matching point $x=c,$ when the distance between two successive values of $x_{i}$ or $\bar{u}_{i}$ becomes greater than the maximal allowed step $h$ \eqref{definition_of_h}. The issue can be solved by applying a "mesh refinement" procedure consisting in adding extra mesh points to fill the "gaps"; this process is rather straightforward and is also described in \cite{Makarov_Dragunov_2019}.

\subsection{Initial guess, mesh selection and choice of point $c\in (a, b)$}
The questions about how to choose point $c$ and meshes $\{x_{i}\},$ $\{u_{i}\},$ as well as how to construct an initial guess for solving the nonlinear system, mentioned above, can be answered simultaneously in scope of the {\it single shooting } procedure described in \cite[Section 3]{Makarov_Dragunov_2019}.

The general idea of the single shooting technique, in its simplest form (see \cite[pp. 132 -- 134]{Ascher_1988}), consists in a gradual approximation of the unknown tangent $\tilde{u}^{\prime}(a)$ based on the results of shooting, that is, the results of solving the corresponding IVP with trial initial conditions. Despite its drawbacks, the  technique can be successfully applied to the boundary value problem \eqref{discretized_equation_straight} -- \eqref{matching_conditions_u}.

Let us fix some maximal discretization step size $h$ \eqref{definition_of_h} and pick some trial tangent value $\tilde{u}^{\prime}(x_{0}),$ $x_{0}=a.$ For the given input data, formula \eqref{u_tilde_through_step_function} allows us to "move" left-to-right and successively calculate values $\tilde{u}(x_{i}), \tilde{u}^{\prime}(x_{i}),$ where $x_{i} = x_{i-1}+h.$ Doing so, on some iteration, we can get $x_{i+1} \geq b$ whereas $x_{i} < b.$ In this case, we set $N_{1} = i+1,$ $x_{N_{1}} = b$ and, depending on how close $u(x_{N_{1}})$ and $u_{b}$ are, we choose different trial tangent value and start the shooting over or stop the process. A more probable scenario, however, provided that solution $u(x)$ has a boundary layer near the right end of interval $[a, b],$ is when for some iteration $j$ we find that the computational cost of evaluating $\tilde{u}(x_{j+1})$ becomes unacceptably high, i.e., the Tailor series, which we use to approximate function $U(x_{j+1}),$ converge extremely slowly (most probably, because of their coefficients having relatively high absolute values and $|\tilde{u}^{\prime}(x_{j})| \gg 1$). In the other words, we face the stiffness. In this case, and here is where the idea of the SI-method comes into the play, we say that point $x_{j}$ is "critical" in the sense that staring from it we cannot "move" left-to-right anymore. We put $c = x_{j},$ transform values $\tilde{u}(c), \tilde{u}^{\prime}(c)$ into $\tilde{x}(\bar{u}_{0}), \tilde{x}^{\prime}(\bar{u}_{0})$ (where $\bar{u}_{0} = \tilde{u}(c)$) using matching equalities \eqref{matching_conditions_u} and proceed by "moving" vertically (bottom-to-top or vise versa, depending on the sign of $\tilde{u}(c)$) using formula \eqref{x_tilde_through_step_function} until we cross the horizontal line $u=u_{b}.$ Depending on where the line was crossed (to the left or to the right from point $x=b$), we adjust the initial tangent and shoot again until the desired accuracy ($|\tilde{x}(u_{b}) - b| \leq \varepsilon$) is achieved.

We do not expect that the single shooting process will provide us precise approximations of functions $\tilde{u}(x),$ $\tilde{x}(u),$ which, otherwise, would be rather inefficient. Instead, we want to get some initial guess for the Newton's method (which is much more efficient once converge) mentioned above in this section. Besides that, the shooting procedure automatically yields us the meshes $\{x_{i}\},$ $\{\bar{u}_{i}\}$ and the "critical" point $c\in (a,b),$ which reflects the maximal "amount" of stiffness we are able to withstand.

Notice that in practice, the criteria of choosing a "critical" point $c$ can be expressed through some maximal acceptable (critical) value $u^{\prime}_{crit.}$ which should not be exceeded by $|\tilde{u}^{\prime}(x)|,$ i.e.:
\begin{equation}\label{criteria_to_calculate_c}
  c = \{x_{i} \; | \; |\tilde{u}^{\prime}(x_{i})| \geq u^{\prime}_{crit.}, \; |\tilde{u}^{\prime}(x_{i-1})| < u^{\prime}_{crit.}\}.
\end{equation}

\section{Numerical examples.}\label{section_numerical_examples}
\subsection{Example 1.} We would like to start with the Troesch's problem \cite{Troesch1976279} (also known as \verb"bvpT23" \cite{Cash_Algo_927})
\begin{equation}\label{Troesch_problem}
  u^{\prime\prime}(x) = \lambda\sinh(\lambda u(x)),\; u(0) = 0,\; u(1) = 1,\; x\in [0,1].
\end{equation}
As it can be easily verified, problem \eqref{Troesch_problem} satisfies conditions \eqref{Intro_Nonlin_cond}, \eqref{lema_1_positiveness_of_nonlinearity}, \eqref{positiveness_of_N_derivative_with_respect_to_x} which means that the results of all the statements proved in the present paper are applicable to the Troesch's problem. The problem was used in \cite{Makarov_Dragunov_2019}, to demonstrate remarkably good accuracy and performance qualities of the SI-method.

This time we want to use the Troesch's problem to examine the results of Theorem \ref{Main_theorem_abour_apptoximation_prorerties_of_the_SI_method}, evaluating constants $\kappa_{S}$ \eqref{asymptotic_estimate_u_minus_u_tilde} and  $\kappa_{I}$ \eqref{asymptotic_estimate_x_minus_x_tilde}. Apparently, in order to do that we need to be able to evaluate the "reference" solution $u(x)$ of problem \eqref{Troesch_problem} by a method (other than the SI-method) which is "trustable" enough and can approximate the solution with an a-priori given accuracy. Our suggestion is to use one of the "standard" numerical BVP solvers from the computing environment {\bf Maple 2016}. Unfortunately, the latter can barely handle the Troesch's problem for $\lambda > 8,$ because of the stiffness. This, however, can be overcome by using the transformation approach proposed in \cite{Chang20103303}, \cite{Chang20103043} and the homotopy approach from \cite{Gen_sol_of_TP}.

Applying the transformation of the unknown solution (see \cite{Chang20103303})
\begin{equation}\label{Transformation_for_Troesch_equation}
  u(x) = \frac{4}{\lambda}\tanh^{-1}(y(x)),
\end{equation}
to problem \eqref{Troesch_problem} we get a significantly less stiff boundary value problem
\begin{equation}\label{Troesch_problem_transformed}
 \left(1-\left(y(x)\right)^{2}\right)y^{\prime\prime}(x)  + 2y(x)\left(y^{\prime}(x)\right)^{2} - \lambda^{2}y(x)\left(1+\left(y(x)\right)^{2}\right) = 0,
\end{equation}
$$y(0) = 0,\; y(1) = \tanh(\lambda/4).$$
Then, introducing a continuation parameter $t$ (see \cite{Gen_sol_of_TP}) we get the "perturbed" problem
\begin{equation}\label{Troesch_problem_transformed_perturbed}
 (1-t)(v^{\prime\prime} - \lambda^{2}v) + t\left(\left(1-v^{2}\right)v^{\prime\prime}(x)  + 2v\left(v^{\prime}(x)\right)^{2} - \lambda^{2}v\left(1+v^{2}\right)\right) = 0,
\end{equation}
$$v = v(t,x),\;  v(t,0) = 0,\; v(t, 1) = \tanh(\lambda/4), \forall t\in [0,1],\; v(1,x) = y(x).$$
The perturbed problem \eqref{Troesch_problem_transformed_perturbed} can be successfully solved by the {\bf Maple 2016} numerical BVP solvers even for sufficiently large values of $\lambda$ (50 and higher)\footnote{We mean calling the subroutine {\bf dsolve} for the problem \eqref{Troesch_problem_transformed_perturbed} with parameters {\bf numeric} and {\bf continuation = t.} For sufficiently small values of parameter {\bf abserr}, one would also need to increase the value of parameter {\bf maxmesh} setting it to $10^{4}$ or higher.}.

Once solution $y(x)$ to problem \eqref{Troesch_problem_transformed} is found, functions $u(x)$ and $u^{\prime}(x)$ can be evaluated using formulas \eqref{Transformation_for_Troesch_equation} and
\begin{equation}\label{Transformation_for_Troesch_equation_derivative}
  u^{\prime}(x) = \frac{4y^{\prime}(x)}{\lambda\left(1-\left(y(x)\right)^{2}\right)}
\end{equation}
respectively.

To evaluate functions $x^{\prime}(u)$ and $x(u)$ we use formulas (compare with \eqref{x_hat_expression})
\begin{equation}\label{Formulas_to_evaluate_inverse_solurions_Troesch_problem}
  x^{\prime}(u) = \left(\left(u^{\prime}(c)\right)^{2} + 2 \left(\cosh(\lambda u) - \cosh(\lambda u(c))\right)\right)^{-\frac{1}{2}},\; x(u) = c+  \int\limits_{u(c)}^{u}x^{\prime}(\eta) d\eta,
\end{equation}
as well as the subroutines for numerical integration available in {\bf Maple 2016}.

\begin{table}
\begin{tabular}{|c|c|c|c|c|c|c|c|c|c|}
  \hline
  $\lambda$ & $c$ & $\tilde{u}(c)$ & $\tilde{u}^{\prime}(c)$ & $N_{1}$ & $N_{2}$& $\kappa^{(0)}_{S}$ & $\kappa^{(1)}_{S}$ & $\kappa^{(0)}_{I}$ & $\kappa_{I}^{(1)}$ \\
  \hline
  1 & 0.589777 & 0.528283 & 1.000001 & 8294 & 6475 & 5.6816e-3 & 1.39645e-2 & 5.68258e-3 & 1.08227e-2 \\
  5 & 0.744141 & 0.192366 & 1.000382 & 8448 & 9195 & 0.19299 & 0.644688 & 0.192962 & 0.432883 \\
  10 & 0.856993 & 9.62509e-2 & 1.000096 & 9557 & 10099 & 0.39043 & 2.62544 & 0.390579 & 1.74209 \\
  15 & 0.903851 & 6.41811e-2 & 1.000327 & 10027 & 10399 & 0.58572 & 5.91021 & 0.585951 & 3.91965 \\
  20 & 0.927852 & 4.81500e-2 & 1.000644 & 10261 & 10550 & 0.78101 & 10.5116 & 0.781248 & 6.96665 \\
  30 & 0.951925 & 3.21270e-2 & 1.001550 & 10504 & 10669 & 1.39351 & 25.8834 & 1.35289 & 19.1854 \\
  \hline
\end{tabular}
\caption{Example 1. Results of numerical experiments for different values of $\lambda$ and $h=10^{-4}.$ }\label{Troesch_table_1}
\end{table}

Table \ref{Troesch_table_1} contains experimental data calculated for different values of parameter $\lambda.$ In all of the cases the "critical" point $c$ was chosen according to formula \eqref{criteria_to_calculate_c} with $u^{\prime}_{crit.} = 1.$ The four rightmost columns of the table contain the values calculated according to formulas:
\begin{eqnarray}\label{Example_1_def_of_kappa_func}
\kappa^{(k)}_{S} = \|\kappa^{(k)}_{S}(x)\|_{[0, c]} = h^{-2} \|u^{(k)}(x) - \tilde{u}^{(k)}(x)\|_{[0, c]},\; k=0,1,\\
\kappa^{(k)}_{I} = \|\kappa^{(k)}_{I}(u)\|_{[\tilde{u}(c), 1]} = h^{-2} \|x^{(k)}(u) - \tilde{x}^{(k)}(u)\|_{[\tilde{u}(c), 1]},\; k=0,1,\nonumber
\end{eqnarray}
where $h = 10^{-4}.$

As we can conclude from the table, the values of $\kappa^{(k)}_{S}$ and $\kappa^{(k)}_{I}$ increase as $\lambda$ increases. This tendency, however, does not hold true for all the values of functions $\kappa^{(k)}_{S}(x)$ and $\kappa^{(k)}_{I}(u)$ on their domains, as it can be seen from Fig. \ref{Example_1_graph_kappa_S}, \ref{Example_1_graph_kappa_S_rel}, \ref{Example_1_graph_kappa_I_rel}. The functions reach their maximums at points $c$ and $\tilde{u}(c)$ respectively and the maximums do increase as parameter $\lambda$ increases. For functions $\kappa^{(k)}_{S}(x),$ $k=0,1$ the behaviour quickly changes to an opposite as we move from point $x=c$ towards the left end of interval $[0,1].$ The same is true for functions $\kappa^{(k)}_{I}(u),$ $k=0,1,$ --- they decrease towards zero on $[\tilde{u}(c), u_{b}],$ with the speed inversely proportional to $\lambda.$ In the other words, the latter means that the accuracy of the SI-method applied to the Troesch's problem decreases near the critical point $x=c$ ($u=\tilde{u}(c)$) and increases near the point $x=0$ ($u=1$) as the problem's stiffness (i.e. parameter $\lambda$) increases.

\begin{figure}[h!]
\centering
\begin{subfigure}[b]{0.47\textwidth}
\begin{tikzpicture}
\begin{axis}[ xlabel={$x$}, ylabel={$\kappa^{(0)}_{S}(x)$}, ymode = log, legend pos=south east ]
\addplot coordinates { (0.1, 1.10710e-11) (0.2, 2.22918e-10) (0.3, 4.47745e-09) (0.4, 8.99321e-08) (0.5, 1.80633e-06) (0.6, 3.62812e-05) (0.7, 0.000728727) (0.8, 0.0146366) (0.9, 0.291483) (0.951854, 1.17176) };
\addplot coordinates { (0.1, 5.88555e-08) (0.2, 4.42852e-07) (0.3, 3.27333e-06) (0.4, 2.41870e-05) (0.5, 0.000178719) (0.6, 0.00132056) (0.7, 0.00975752) (0.8, 0.0720158) (0.9, 0.501698) (0.927781, 0.781005) };
\addplot coordinates { (0.1, 3.84229e-06) (0.2, 1.80773e-05) (0.3, 8.12081e-05) (0.4, 0.000363992) (0.5, 0.00163131) (0.6, 0.00731087) (0.7, 0.0327516) (0.8, 0.145594) (0.9, 0.563241) (0.90378, 0.585724) };
\addplot coordinates { (0.1, 0.000207222) (0.2, 0.000639522) (0.3, 0.00176644) (0.4, 0.00481191) (0.5, 0.0130822) (0.6, 0.0355273) (0.7, 0.0958794) (0.8, 0.247855) (0.856922, 0.39043) };
\addplot coordinates { (0.1, 0.00580936) (0.2, 0.0130954) (0.3, 0.0236939) (0.4, 0.0402026) (0.5, 0.0663688) (0.6, 0.106922) (0.7, 0.164355) (0.744071, 0.19299) };
\legend{$\lambda = 30$, $\lambda = 20$, $\lambda = 15$, $\lambda = 10$, $\lambda = 5$}
\end{axis}
\end{tikzpicture}
\end{subfigure}
\begin{subfigure}[b]{0.47\textwidth}
\begin{tikzpicture}
\begin{axis}[ xlabel={$x$}, ylabel={$\kappa^{(1)}_{S}(x)$}, ymode = log, legend pos=south east ]
\addplot coordinates { (0, 3.31537e-11) (0.1, 3.33780e-10) (0.2, 6.68762e-09) (0.3, 1.34324e-07) (0.4, 2.69796e-06) (0.5, 5.41900e-05) (0.6, 0.00108844) (0.7, 0.0218618) (0.8, 0.439078) (0.9, 8.59636) (0.951854, 23.6634) };
\addplot coordinates {(0, 3.24553e-07) (0.1, 1.22103e-06) (0.2, 8.86298e-06) (0.3, 6.54675e-05) (0.4, 0.00048374) (0.5, 0.00357438) (0.6, 0.0264113) (0.7, 0.195142) (0.8, 1.43693) (0.9, 8.86684) (0.927781, 10.5116) };
\addplot coordinates { (0, 2.70676e-05) (0.1, 6.36740e-05) (0.2, 0.000272507) (0.3, 0.00121842) (0.4, 0.00545995) (0.5, 0.0244696) (0.6, 0.109658) (0.7, 0.490849) (0.8, 2.14699) (0.9, 5.96668) (0.90378, 5.91021) };
\addplot coordinates { (0, 0.00176329) (0.1, 0.0027209) (0.2, 0.00663385) (0.3, 0.0177521) (0.4, 0.0481494) (0.5, 0.130793) (0.6, 0.354464) (0.7, 0.942965) (0.8, 2.19703) (0.856922, 2.62544) };
\addplot coordinates { (0, 0.0557479) (0.1, 0.062842) (0.2, 0.0858715) (0.3, 0.130355) (0.4, 0.206048) (0.5, 0.325522) (0.6, 0.491265) (0.7, 0.642304) (0.744071, 0.644688) };
\legend {$\lambda = 30$, $\lambda = 20$, $\lambda = 15$, $\lambda = 10$, $\lambda = 5$}
\end{axis}
\end{tikzpicture}
\end{subfigure}
\caption{Example 1. Graphs of functions $\kappa^{(k)}_{S}(x),$ $k=0,1$ \eqref{Example_1_def_of_kappa_func} that correspond to different values of parameter $\lambda;$ $h=10^{-4}$} \label{Example_1_graph_kappa_S}
\end{figure}
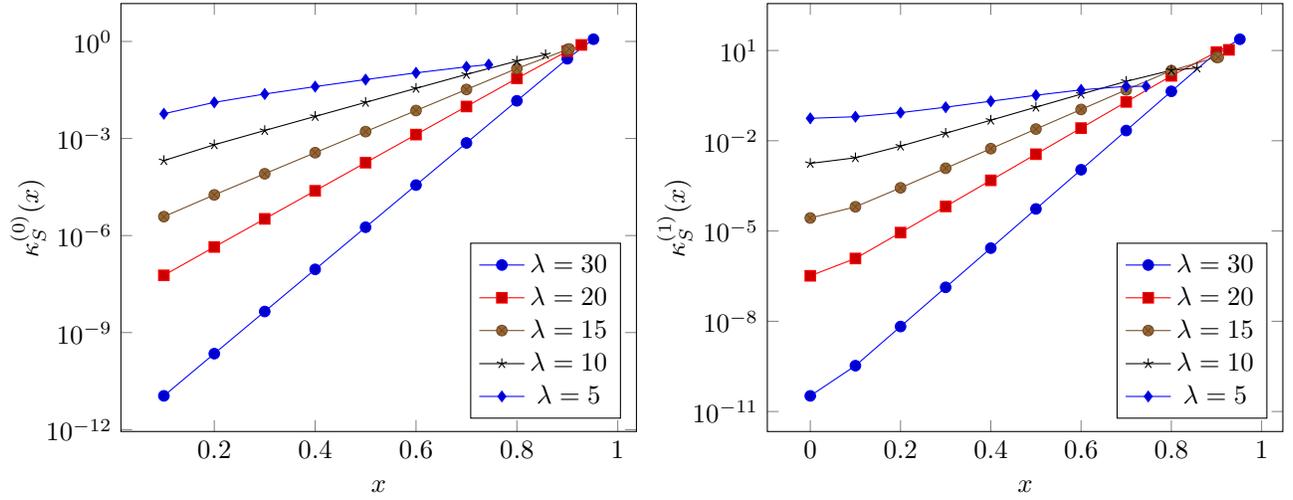

\begin{figure}[h!]
\centering
\begin{subfigure}[b]{0.47\textwidth}
\begin{tikzpicture}
\begin{axis}[ xlabel={$x$}, ylabel={$\kappa^{(0)}_{S}(x)/x$}, legend pos=south west ]
\addplot coordinates { (0.1, 44.287) (0.2, 44.287) (0.3, 44.287) (0.4, 44.287) (0.5, 44.287) (0.6, 44.287) (0.7, 44.287) (0.8, 44.286) (0.9, 43.8731) (0.951854, 36.5538) };
\addplot coordinates { (0.1, 19.6845) (0.2, 19.6845) (0.3, 19.6845) (0.4, 19.6845) (0.5, 19.6845) (0.6, 19.6845) (0.7, 19.6841) (0.8, 19.6592) (0.9, 18.4233) (0.927781, 16.2442) };
\addplot coordinates { (0.1, 11.0728) (0.2, 11.0728) (0.3, 11.0728) (0.4, 11.0728) (0.5, 11.0728) (0.6, 11.0725) (0.7, 11.0675) (0.8, 10.9694) (0.9, 9.31743) (0.90378, 9.1362) };
\addplot coordinates { (0.1, 4.92076) (0.2, 4.92076) (0.3, 4.92074) (0.4, 4.92064) (0.5, 4.91992) (0.6, 4.91459) (0.7, 4.87587) (0.8, 4.61294) (0.856922, 4.05937) };
\addplot coordinates { (0.1, 1.21839) (0.2, 1.21779) (0.3, 1.21599) (0.4, 1.2109) (0.5, 1.19718) (0.6, 1.16164) (0.7, 1.07308) (0.744071, 1.00361) };
\legend{$\lambda = 30$, $\lambda = 20$, $\lambda = 15$, $\lambda = 10$, $\lambda = 5$}
\end{axis}
\end{tikzpicture}
\end{subfigure}
\begin{subfigure}[b]{0.47\textwidth}
\begin{tikzpicture}
\begin{axis}[ xlabel={$x$}, ylabel={$\kappa^{(1)}_{S}(x)/x$}, legend pos=south west ]
\addplot coordinates { (0, 44.287) (0.1, 44.287) (0.2, 44.287) (0.3, 44.287) (0.4, 44.287) (0.5, 44.287) (0.6, 44.287) (0.7, 44.287) (0.8, 44.2839) (0.9, 43.0586) (0.951854, 23.6833) };
\addplot coordinates { (0, 19.6845) (0.1, 19.6845) (0.2, 19.6845) (0.3, 19.6845) (0.4, 19.6845) (0.5, 19.6845) (0.6, 19.6845) (0.7, 19.6831) (0.8, 19.6087) (0.9, 16.0809) (0.927781, 10.5215) };
\addplot coordinates { (0, 11.0728) (0.1, 11.0728) (0.2, 11.0728) (0.3, 11.0728) (0.4, 11.0728) (0.5, 11.0727) (0.6, 11.072) (0.7, 11.0571) (0.8, 10.7661) (0.9, 6.36013) (0.90378, 5.91531) };
\addplot coordinates { (0, 4.92076) (0.1, 4.92076) (0.2, 4.92075) (0.3, 4.92071) (0.4, 4.92042) (0.5, 4.91825) (0.6, 4.90228) (0.7, 4.78764) (0.8, 4.04022) (0.856922, 2.62727) };
\addplot coordinates { (0, 1.21852) (0.1, 1.21811) (0.2, 1.21627) (0.3, 1.2107) (0.4, 1.19534) (0.5, 1.1551) (0.6, 1.05299) (0.7, 0.81719) (0.744071, 0.644696) };
\legend{$\lambda=30$, $\lambda = 20$, $\lambda = 15$, $\lambda = 10$, $\lambda = 5$}
\end{axis}
\end{tikzpicture}
\end{subfigure}
\caption{Example 1. Graphs of functions $\kappa^{(k)}_{S}(x)/x,$ $k=0,1$ \eqref{Example_1_def_of_kappa_func} that correspond to different values of parameter $\lambda;$ $h = 10^{-4}$}\label{Example_1_graph_kappa_S_rel}
\end{figure}
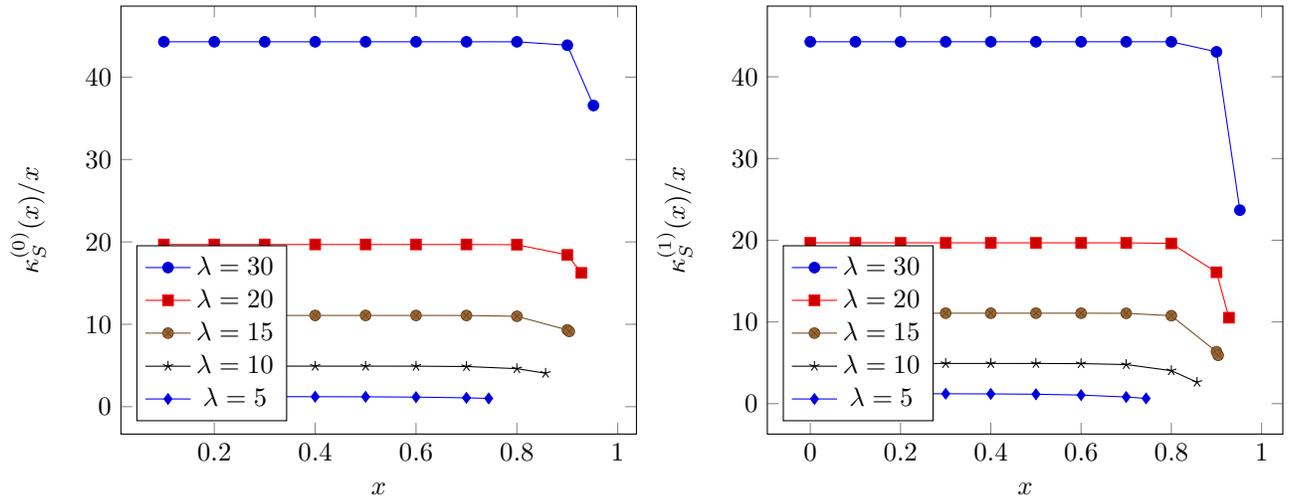

\begin{figure}[h!]
\centering
\begin{subfigure}[b]{0.47\textwidth}
\begin{tikzpicture}
\begin{axis}[ xlabel={$u$}, ylabel={$\kappa_{I}(u)/u$}, ymode = log, legend pos=south west ]
\addplot coordinates { (0.032127, 1.2308) (0.10008, 0.121863) (0.200003, 0.00249063) (0.300004, 4.12388e-05) (0.400004, 6.10410e-07) (0.500004, 8.47560e-09) (0.600004, 1.12988e-10) (0.700004, 1.46440e-12) (0.800004, 1.85896e-14) (0.900004, 2.29497e-16) };
\addplot coordinates { (0.04815, 0.841997) (0.100036, 0.28312) (0.200092, 0.0229724) (0.300099, 0.00165775) (0.4001, 0.000109312) (0.5001, 6.79188e-06) (0.6001, 4.05583e-07) (0.7001, 2.35508e-08) (0.8001, 1.33604e-09) (0.9001, 7.09001e-11) };
\addplot coordinates { (0.0641811, 0.648283) (0.100015, 0.390258) (0.200013, 0.0619784) (0.300028, 0.00906869) (0.400031, 0.00124736) (0.500031, 0.000163246) (0.600032, 2.05955e-05) (0.700032, 2.52655e-06) (0.800032, 3.00501e-07) (0.900032, 3.18942e-08) };
\addplot coordinates { (0.0962509, 0.455755) (0.100052, 0.445138) (0.200087, 0.147113) (0.300038, 0.041887) (0.400052, 0.0116382) (0.500056, 0.00314911) (0.600058, 0.000829399) (0.700059, 0.000211981) (0.800059, 5.13379e-05) (0.900059, 1.03147e-05) };
\addplot coordinates { (0.192366, 0.259308) (0.20002, 0.251762) (0.300081, 0.14394) (0.400055, 0.0754622) (0.500086, 0.0387809) (0.600005, 0.0195802) (0.700014, 0.00949563) (0.800019, 0.0042055) (0.900022, 0.00143895) };
\legend{$\lambda = 30$, $\lambda = 20$, $\lambda = 15$, $\lambda = 10$, $\lambda = 5$}
\end{axis}
\end{tikzpicture}
\end{subfigure}
\begin{subfigure}[b]{0.47\textwidth}
\begin{tikzpicture}
\begin{axis}[ xlabel={$u$}, ylabel={$\kappa^{(1)}_{I}(u)/u$}, ymode = log, legend pos=south west ]
\addplot coordinates { (0.032127, 15.699) (0.10008, 19.1053) (0.200003, 1.99578) (0.300004, 0.154583) (0.400004, 0.0104653) (0.500004, 0.000658985) (0.600004, 3.96772e-05) (0.700004, 2.31736e-06) (0.800004, 1.32400e-07) (0.900004, 7.43942e-09) (1, 4.12632e-10) };
\addplot coordinates { (0.04815, 6.97114) (0.100036, 15.2889) (0.200092, 4.20939) (0.300099, 0.884947) (0.4001, 0.164011) (0.5001, 0.0282349) (0.6001, 0.00463973) (0.7001, 0.00073873) (0.8001, 0.000114973) (0.9001, 1.75896e-05) (1, 2.66021e-06) };
\addplot coordinates { (0.0641811, 3.92093) (0.100015, 9.85572) (0.200013, 4.77661) (0.300028, 1.63076) (0.400031, 0.4982) (0.500031, 0.14192) (0.600032, 0.0385885) (0.700032, 0.0101587) (0.800032, 0.00261246) (0.900032, 0.000660079) (1, 0.000164568) };
\addplot coordinates { (0.0962509, 1.74226) (0.100052, 2.25349) (0.200087, 3.81838) (0.300038, 2.12165) (0.400052, 1.05246) (0.500056, 0.492667) (0.600058, 0.221307) (0.700059, 0.0963991) (0.800059, 0.0410196) (0.900059, 0.0171417) (1, 0.00706547) };
\addplot coordinates { (0.192366, 0.433049) (0.20002, 0.560441) (0.300081, 1.08876) (0.400055, 0.951236) (0.500086, 0.726318) (0.600005, 0.529363) (0.700014, 0.376331) (0.800019, 0.26283) (0.900022, 0.180955) (1, 0.123103) };
\legend{$\lambda = 30$, $\lambda = 20$, $\lambda = 15$, $\lambda = 10$, $\lambda = 5$}
\end{axis}
\end{tikzpicture}
\end{subfigure}
\caption{Example 1. Graphs of functions $\kappa^{(k)}_{I}(u)/u,$ $k=0,1$ \eqref{Example_1_def_of_kappa_func} that correspond to different values of parameter $\lambda;$ $h = 10^{-4}$}
\label{Example_1_graph_kappa_I_rel}
\end{figure}
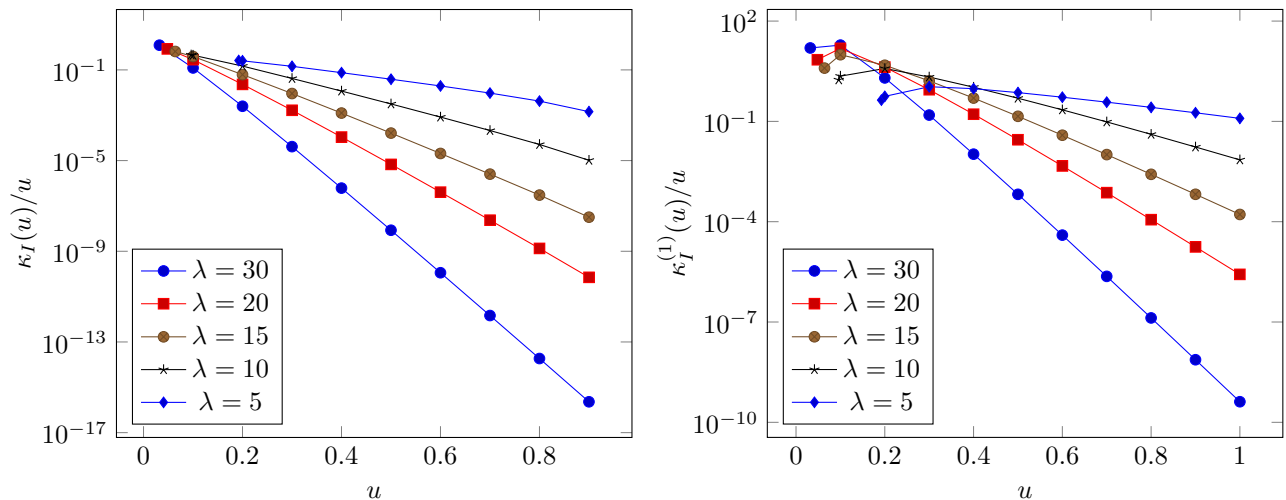

It is worthwhile to mention, that the graphs of quotients $\kappa^{(k)}_{S}(x)/x,$ $k=0,1$ depicted on Fig. \ref{Example_1_graph_kappa_S_rel}, clearly show that, despite the increase of accuracy of the SI-method near point $x=0,$ the overall number of significant digits that the method can provide us when approximating functions $u(x),$ $u^{\prime}(x)$ actually decreases as the Troesch's problem becomes stiffer.

Finally, to check the conclusion of Theorem \ref{Main_theorem_abour_apptoximation_prorerties_of_the_SI_method} about the approximation order of the SI-method (presented in the current paper) with respect to $h$ \eqref{definition_of_h} we need to demonstrate that functions $\kappa^{(k)}_{S}(x) = \kappa^{(k)}_{S}(x, h)$ and $\kappa^{(k)}_{I}(x) = \kappa^{(k)}_{I}(x, h)$ are bounded for $h$ sufficiently small. This is done by means of Fig. \ref{figure_kappa_S_bounded}, \ref{figure_kappa_I_bounded} exhibiting  graphs of functions
\begin{equation}\label{function_to_show_that_kappas_are_bounded}
  \frac{|\kappa^{(k)}_{S}(x, h) - \kappa^{(k)}_{S}(x, 10^{-6})|}{x}, \; \frac{|\kappa^{(k)}_{I}(u, h) - \kappa^{(k)}_{I}(u, 10^{-6})|}{u}\footnote{The "trick" with subtracting $\kappa^{(k)}_{S}(x, 10^{-6})$ and $\kappa^{(k)}_{I}(u, 10^{-6})$ serves the only purpose to make the corresponding graphs more distinguishable on the figures.}
\end{equation}
for different values of $h$ and for $\lambda = 30,$ $k=0,1.$ The graphs clearly indicate that functions \eqref{function_to_show_that_kappas_are_bounded} converge uniformly as $h$ tends to 0, which imply their uniform boundedness for sufficiently small values of $h.$

\begin{figure}[h!]
\centering
\begin{subfigure}[b]{0.47\textwidth}
\begin{tikzpicture}
\begin{axis}[ xlabel={$x$}, ylabel={}, legend pos=south west, ymode = log   ]
\addplot coordinates { (0.1, 52.1581-43.9604) (0.2, 52.1581-43.9604) (0.3, 52.1581-43.9604) (0.4, 52.1581-43.9604) (0.5, 52.1581-43.9604) (0.6, 52.1581-43.9604) (0.7, 52.1581-43.9604) (0.8, 52.1574-43.9594) (0.9, 51.8704-43.5578) };
\addplot coordinates { (0.1, 44.287-43.9604) (0.2, 44.287-43.9604) (0.3, 44.287-43.9604) (0.4, 44.287-43.9604) (0.5, 44.287-43.9604) (0.6, 44.287-43.9604) (0.7, 44.287-43.9604) (0.8, 44.286-43.9594) (0.9, 43.8731-43.5578)};
\addplot coordinates { (0.1, 44.287-44.2739) (0.2, 44.287-44.2739) (0.3, 44.287-44.2739) (0.4, 44.287-44.2739) (0.5, 44.287-44.2739) (0.6, 44.287-44.2739) (0.7, 44.287-44.2739) (0.8, 44.286-44.2728) (0.9, 43.8731-43.8587)};
\addplot coordinates { (0.1, 44.2769-44.2739) (0.2, 44.2769-44.2739) (0.3, 44.2769-44.2739) (0.4, 44.2769-44.2739) (0.5, 44.2769-44.2739) (0.6, 44.2769-44.2739) (0.7, 44.2769-44.2739) (0.8, 44.2758-44.2728) (0.9, 43.8615-43.8587)};
\legend{$h=10^{-2}$, $h=10^{-3}$, $h=10^{-4}$, $h=10^{-5}$, $h=10^{-6}$}
\end{axis}
\end{tikzpicture}
\end{subfigure}
\begin{subfigure}[b]{0.47\textwidth}
\begin{tikzpicture}
\begin{axis}[ xlabel={$x$}, ylabel={}, legend pos=south west, ymode = log ]
\addplot coordinates { (0, 52.1581-43.9604) (0.1, 52.1581-43.9604) (0.2, 52.1581-43.9604) (0.3, 52.1581-43.9604) (0.4, 52.1581-43.9604) (0.5, 52.1581-43.9604) (0.6, 52.1581-43.9604) (0.7, 52.1581-43.9604) (0.8, 52.1555-43.9573) (0.9, 51.0808-42.7485) };
\addplot coordinates { (0, 44.287-43.9604) (0.1, 44.287-43.9604) (0.2, 44.287-43.9604) (0.3, 44.287-43.9604) (0.4, 44.287-43.9604) (0.5, 44.287-43.9604) (0.6, 44.287-43.9604) (0.7, 44.287-43.9604) (0.8, 44.2839-43.9573) (0.9, 43.0586-42.7485) };
\addplot coordinates { (0, 44.287-44.2739) (0.1, 44.287-44.2739) (0.2, 44.287-44.2739) (0.3, 44.287-44.2739) (0.4, 44.287-44.2739) (0.5, 44.287-44.2739) (0.6, 44.287-44.2739) (0.7, 44.287-44.2738) (0.8, 44.2839-44.2707) (0.9, 43.0586-43.0435)};
\addplot coordinates { (0, 44.2769-44.2739) (0.1, 44.2769-44.2739) (0.2, 44.2769-44.2739) (0.3, 44.2769-44.2739) (0.4, 44.2769-44.2739) (0.5, 44.2769-44.2739) (0.6, 44.2769-44.2739) (0.7, 44.2769-44.2738) (0.8, 44.2737-44.2707) (0.9, 43.0463-43.0435)};
\legend{$h=10^{-2}$, $h=10^{-3}$, $h=10^{-4}$, $h=10^{-5}$, $h=10^{-6}$}
\end{axis}
\end{tikzpicture}
\end{subfigure}
\caption{Example 1. Graphs of functions $|\kappa^{(k)}_{S}(x, h) - \kappa^{(k)}_{S}(x, 10^{-6})|/x$ for different values of $h$ \eqref{definition_of_h}, $\lambda = 30$ (left figure: $k = 0;$ right figure: $k = 1$).}\label{figure_kappa_S_bounded}
\end{figure}
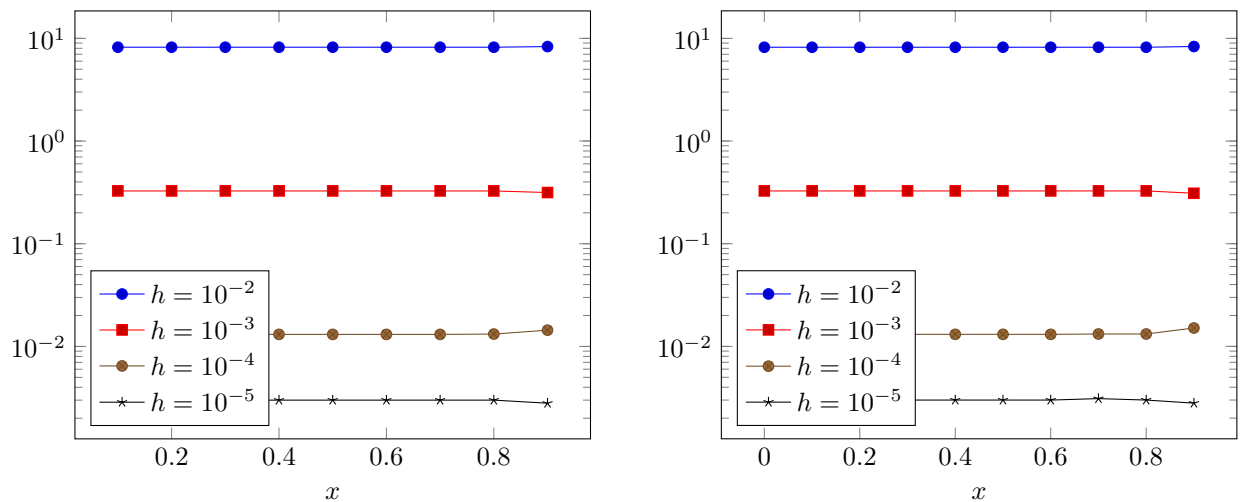

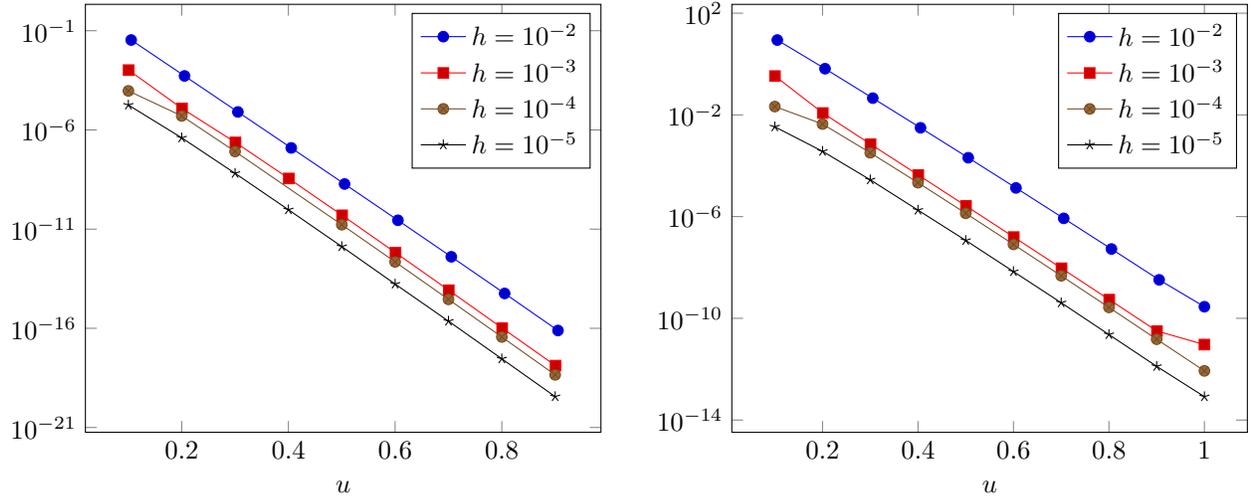
\begin{figure}[h!]
\centering
\begin{subfigure}[b]{0.47\textwidth}
\begin{tikzpicture}
\begin{axis}[ xlabel={$u$}, ylabel={}, ymode = log ]
\addplot coordinates {(0.105266, 0.157039-0.122901) (0.205266, 0.0030062-0.00247859) (0.305266, 4.90757e-05-4.10034e-05) (0.405266, 7.31000e-07-6.06825e-07) (0.505266, 1.03076e-08-8.42583e-09) (0.605266, 1.40216e-10-1.12333e-10) (0.705266, 1.85982e-12-1.45609e-12) (0.805266, 2.42070e-14-1.84867e-14) (0.905266, 3.05744e-16-2.28184e-16) };

\addplot coordinates {(0.100363, 0.122901-0.121863) (0.200618, 0.00249063-0.00247859) (0.30063, 4.12388e-05-4.10034e-05) (0.400631, 6.10410e-07-6.06825e-07) (0.500631, 8.47560e-09-8.42583e-09) (0.600631, 1.12988e-10-1.12333e-10) (0.700631, 1.46440e-12-1.45609e-12) (0.800631, 1.85896e-14-1.84867e-14) (0.900631, 2.29497e-16-2.28184e-16) };

\addplot coordinates { (0.10008, 0.121956-0.121863) (0.200003, 0.00249063-0.00248555) (0.300004, 4.12388e-05-4.11569e-05) (0.400004, 6.09204e-07-6.10410e-07) (0.500004, 8.47560e-09-8.45890e-09) (0.600004, 1.12988e-10-1.12765e-10) (0.700004, 1.46440e-12-1.46152e-12) (0.800004, 1.85896e-14-1.85530e-14) (0.900004, 2.29497e-16-2.29046e-16) };

\addplot coordinates { (0.100002, 0.121956-0.121938) (0.200002, 0.00248555-0.00248515) (0.300002, 4.11569e-05-4.11504e-05) (0.400002, 6.09204e-07-6.09109e-07) (0.500002, 8.45890e-09-8.45758e-09) (0.600002, 1.12765e-10-1.12748e-10) (0.700002, 1.46152e-12-1.46129e-12) (0.800002, 1.85530e-14-1.85501e-14) (0.900002, 2.29046e-16-2.29010e-16) };

\legend{$h=10^{-2}$, $h=10^{-3}$, $h=10^{-4}$, $h=10^{-5}$, $h=10^{-6}$}
\end{axis}
\end{tikzpicture}
\end{subfigure}
\begin{subfigure}[b]{0.47\textwidth}
\begin{tikzpicture}
\begin{axis}[ xlabel={$u$}, ylabel={}, ymode = log ]
\addplot coordinates { (0.105266, 28.231-19.443) (0.205266, 2.66327-2.00765) (0.305266, 0.200917-0.155289) (0.405266, 0.0136185-0.0105086) (0.505266, 0.000868521-0.00066161) (0.605266, 5.32749e-05-3.98345e-05) (0.705266, 3.18088e-06-2.32668e-06) (0.805266, 1.86202e-07-1.32945e-07) (0.905266, 1.07365e-08-7.47101e-09) (1, 7.08059e-10-4.21907e-10)};
\addplot coordinates { (0.100363, 19.443-19.1053) (0.200618, 2.00765-1.99578) (0.30063, 0.155289-0.154583) (0.400631, 0.0105086-0.0104653) (0.500631, 0.00066161-0.000658985) (0.600631, 3.98345e-05-3.96772e-05) (0.700631, 2.32668e-06-2.31736e-06) (0.800631, 1.32945e-07-1.32400e-07) (0.900631, 7.47101e-09-7.43942e-09) (1, 4.21907e-10-4.12632e-10)};
\addplot coordinates { (0.10008, 19.1053-19.084) (0.200003, 1.99578-1.99141) (0.300004, 0.154583-0.154259) (0.400004, 0.0104653-0.0104437) (0.500004, 0.000658985-0.000657635) (0.600004, 3.96772e-05-3.95963e-05) (0.700004, 2.31736e-06-2.31266e-06) (0.800004, 1.32400e-07-1.32131e-07) (0.900004, 7.43942e-09-7.42436e-09) (1, 4.12632e-10-4.11776e-10) };
\addplot coordinates { (0.100002, 19.084-19.0806) (0.200002, 1.99141-1.99104) (0.300002, 0.154259-0.154231) (0.400002, 0.0104437-0.0104419) (0.500002, 0.000657635-0.000657519) (0.600002, 3.95963e-05-3.95894e-05) (0.700002, 2.31266e-06-2.31225e-06) (0.800002, 1.32131e-07-1.32108e-07) (0.900002, 7.42436e-09-7.42307e-09) (1, 4.11776e-10-4.11693e-10) };
\legend{$h=10^{-2}$, $h=10^{-3}$, $h=10^{-4}$, $h=10^{-5}$, $h=10^{-6}$}
\end{axis}
\end{tikzpicture}
\end{subfigure}
\caption{Example 1. Graphs of functions $|\kappa^{(k)}_{I}(u, h) - \kappa^{(k)}_{I}(u, 10^{-6})|/u$ for different values of $h$ \eqref{definition_of_h}, $\lambda = 30$ (left figure: $k = 0;$ right figure: $k = 1$).}\label{figure_kappa_I_bounded}
\end{figure}

\subsection{Example 2.} As the second example we want to examine the SI-method with a problem that actually does not fit into the pattern \eqref{Intro_Equation}, \eqref{Intro_boundary_conditions}:
\begin{equation}\label{bvp_t_21}
  \xi u^{\prime\prime}(x) = (u(x) + 1)u(x) - \exp\left(-2x/\sqrt{\xi}\right),\; x\in [0,1].
\end{equation}
$$u(0) = 1,\; u(1) = \exp\left(-1/\sqrt{\xi}\right),$$

\begin{figure}[h!]
\centering
\begin{tikzpicture}
\begin{axis}[ xlabel={$x$}, ylabel={$u(x)$},legend pos=north east ]
\addplot [domain=0:1, samples=100, line width = 1pt, blue] {exp(-x/sqrt(0.1))};
\addplot [domain=0:1, samples=100, line width = 1pt, green] {exp(-x/sqrt(0.01))};
\addplot [domain=0:1, samples=100, line width = 1pt, yellow] {exp(-x/sqrt(0.001))};
\addplot [domain=0:1, samples=300, line width = 1pt, orange] {exp(-x/sqrt(0.0001))};
\addplot [domain=0:1, samples=600, line width = 1pt, red] {exp(-x/sqrt(0.00001))};
\legend{$\xi = 10^{-1}$, $\xi = 10^{-2}$, $\xi = 10^{-3}$, $\xi = 10^{-4}$, $\xi = 10^{-5}$}
\end{axis}
\end{tikzpicture}

\caption{Example 2. Graphs of solution $u(x)$ \eqref{example_2_bvp_t_21_solution} to problem \eqref{bvp_t_21} for different values of parameter $\xi.$}\label{fig_bvp_t_solution_graphs}
\end{figure}
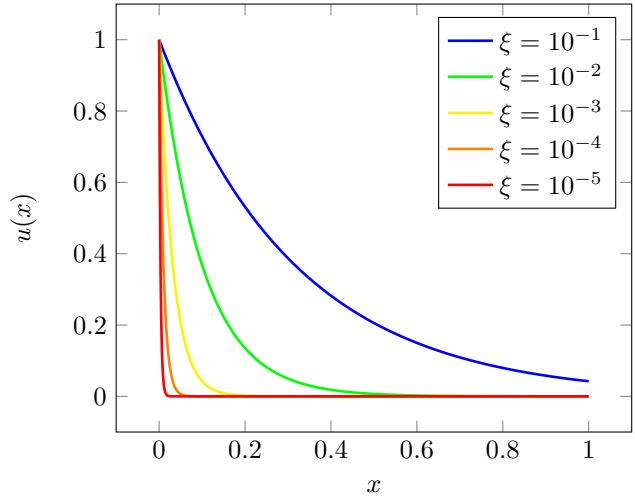
The problem is also known as \verb"bvpT21" \cite{Cash_Algo_927}. It has an explicit-form solution:
\begin{equation}\label{example_2_bvp_t_21_solution}
  u(x) = \exp\left(-x/\sqrt{\xi}\right).
\end{equation}

The graphs on Fig. \ref{fig_bvp_t_solution_graphs} show that for small values of parameter $\xi > 0$ problem \eqref{bvp_t_21} becomes stiff, which, in this particular case, amounts to its solution having a boundary layer near point $x = 0.$

The general idea of the SI-method suggests that for the case of \verb"bvpT21" the "hybrid" equations \eqref{c_family_problem_straight}, \eqref{c_family_problem_inverse} should be rewritten as
$$\mathfrak{u}^{\prime\prime}(x) = \mathcal{N}(\mathfrak{u}(x), x), \; x\in [c, 1],\; \mathfrak{u}(1) = \exp\left(-1/\sqrt{\xi}\right),\; \mathfrak{u}(c) \neq 1,$$
$$\mathfrak{x}^{\prime\prime}(x) = -\mathcal{N}(u, \mathfrak{x}(u))\left(\mathfrak{x}^{\prime}(u)\right)^{3}, \; u\in [\mathfrak{u}(c), 1],\; \mathfrak{x}(1) = 0,$$
and the "critical" point $c = c(\xi)$ should be chosen so that interval $[0, c]$ contains the boundary layer of solution $u(x) = u(x, \xi).$ From the SI-method's point of view, there is almost no difference between the Troesch's problem and \verb"bvpT21" if the latter is considered with respect to a new independent variable $t = 1-x.$

The non-uniformity of problem \eqref{bvp_t_21} does require some changes in the SI-method's implementation as compared to what was described in Section \ref{section_implementation_aspects}. In order to ensure the second order of approximation with respect to $h$ we need to substitute the uniform equations \eqref{discretized_equation_straight} with the corresponding non-uniform ones. The latter automatically entails the necessity to re-define the {\it straight step function } \eqref{SI_method_definition_of_U} as  $U(s) = U(s, A, B, C, D, E, F)$ satisfying the IVP
 $$ U^{\prime\prime}(s) = \left(As + B\right)U(s) + Es+F, \;\; U(0) = D,\; U^{\prime}(0) = C.$$

\begin{table}
\begin{tabular}{|c|c|c|c|c|c|c|c|c|c|}
  \hline
  $\xi$ & $c$ & $\tilde{u}(c)$ & $\tilde{u}^{\prime}(c)$ & $N_{S}$ & $N_{I}$& $\kappa^{(0)}_{S}$ & $\kappa^{(1)}_{S}$ & $\kappa^{(0)}_{I}$ & $\kappa_{I}^{(1)}$ \\
  \hline
   7.5e-2 & 0.354623 & 0.273924 & -1.000232 & 7098 & 8203 & 0.070653 & 0.671826 & 0.080263 & 0.413812 \\
   1e-2 & 0.230238 & 0.100021 & -1.000206 & 7985 & 9461 & 0.397041 & 5.281615 & 0.402318 & 1.841309 \\
   1e-3 & 0.109198 & 0.031646 & -1.000726 & 9065 & 9901 & 1.609073 & 55.31288 & 1.613678 & 22.08847 \\
   1e-4 & 0.046048 & 0.010003 & -1.000370 & 9607 & 10042 & 5.408764 & 554.6440 & 5.447916 & 232.6789 \\
   1e-5 & 0.018149 & 3.21756e-3 & -1.017484 & 10106 & 10084 & 17.90946 & 5708.525 & 18.00913 & 2409.302 \\
  \hline
\end{tabular}
\caption{Example 2. Results of numerical experiments for different values of $\xi$ and $h=10^{-4}.$ }\label{table_bvpT21}
\end{table}

\begin{figure}[h!]
\centering
\begin{subfigure}[b]{0.47\textwidth}
\begin{tikzpicture}
\begin{axis}[ xlabel={$x$}, ylabel={$\kappa^{(0)}_{S}(x)$}, ymode = log, legend pos=south west ]
\addplot coordinates { (0.9, 2.511e-4) (0.8, 7.761e-4) (0.7, 2.153e-3) (0.6, 5.937e-3) (0.5, 1.666e-2) (0.4, 4.915e-2) (0.3, 1.608e-1) (0.23, 3.97e-1)};
\addplot coordinates { (0.9, 1.98e-11) (0.8, 4.688e-10) (0.7, 1.107e-8) (0.6, 2.616e-7) (0.5, 6.181e-6) (0.4, 1.46e-4) (0.3, 3.45e-3) (0.2, 8.21e-2) (0.109, 1.609)};
\addplot coordinates { (0.9, 4.363e-37) (0.8, 9.605e-33) (0.7, 2.114e-28) (0.6, 4.655e-24) (0.5, 1.024e-19) (0.4, 2.256e-15) (0.3, 4.953e-11) (0.2, 1.090e-6) (0.1, 0.0240) (0.0461, 5.408)};
\addplot coordinates { (0.9, 1.448e-120) (0.8, 7.654e-107) (0.7, 4.174e-93) (0.6, 2.251e-79) (0.5, 1.228e-65) (0.4, 6.698e-52) (0.3, 3.653e-38) (0.2, 1.970e-24) (0.1, 1.074e-10) (0.0182, 17.909)};
\legend {$\xi = 10^{-2}$, $\xi = 10^{-3}$, $\xi = 10^{-4}$, $\xi = 10^{-5}$}
\end{axis}
\end{tikzpicture}
\end{subfigure}
\begin{subfigure}[b]{0.47\textwidth}
\begin{tikzpicture}
\begin{axis}[ xlabel={$x$}, ylabel={$\kappa^{(1)}_{S}(x)$}, ymode = log, legend pos=south west ]
\addplot coordinates { (1, 2.135e-3) (0.9, 3.299e-3) (0.8, 8.071e-3) (0.7, 2.179e-2) (0.6, 6.052e-2) (0.5, 1.749e-1) (0.4, 5.519e-1) (0.3, 2.017) (0.23, 5.281)};
\addplot coordinates { (1, 5.312e-11) (0.9, 6.286e-10) (0.8, 1.482e-8) (0.7, 3.502e-7) (0.6, 8.274e-6) (0.5, 1.954e-4) (0.4, 4.618e-3) (0.3, 0.1091) (0.2, 2.617) (0.109, 55.31)};
\addplot coordinates { (1, 3.924e-39) (0.9, 4.363e-35) (0.8, 9.605e-31) (0.7, 2.114e-26) (0.6, 4.655e-22) (0.5, 1.024e-17) (0.4, 2.256e-13) (0.3, 4.953e-9) (0.2, 1.09e-4) (0.1, 2.401) (0.0461, 554.64)};
\addplot coordinates { (1, 1.646e-131) (0.9, 4.579e-118) (0.8, 2.420e-104) (0.7, 1.320e-90) (0.6, 7.119e-77) (0.5, 3.883e-63) (0.4, 2.118e-49) (0.3, 1.155e-35) (0.2, 6.231e-22) (0.1, 3.398e-8) (0.0182, 5708.525)};
\legend {$\xi = 10^{-2}$, $\xi = 10^{-3}$, $\xi = 10^{-4}$, $\xi = 10^{-5}$}
\end{axis}
\end{tikzpicture}
\end{subfigure}
\caption{Example 2. Graphs of functions $\kappa^{(k)}_{S}(x),$ $k=0,1$ that correspond to different values of parameter $\xi;$ $h=10^{-4}.$} \label{Example_2_graph_kappa_S}
\end{figure}
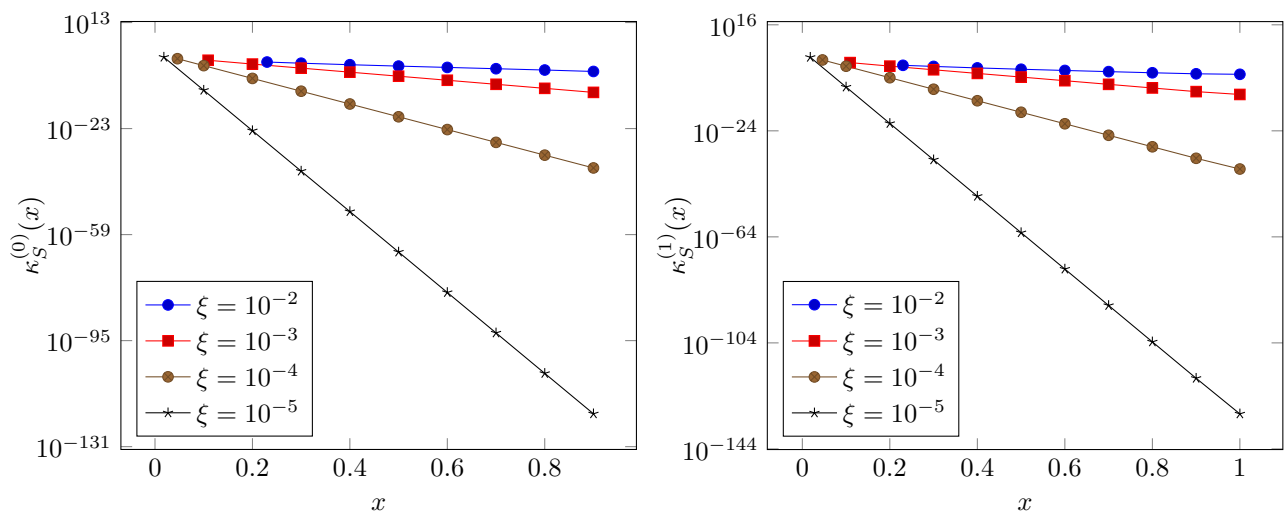

\begin{figure}[h!]
\centering
\begin{subfigure}[b]{0.47\textwidth}
\begin{tikzpicture}
\begin{axis}[ xlabel={$x$}, ylabel={$\kappa^{(0)}_{I}(x)$}, ymode = log, legend pos=north east ]
\addplot coordinates { (0.1, 3.973e-1) (0.2, 2.612e-1) (0.3, 1.474e-1) (0.4, 8.840e-2) (0.5, 5.511e-2) (0.6, 3.471e-2) (0.7, 2.133e-2) (0.8, 1.201e-2) (0.9, 5.199e-3)};
\addplot coordinates { (0.0316, 1.611) (0.1, 0.587) (0.2, 0.194) (0.3, 0.0908) (0.4, 0.0495) (0.5, 0.0291) (0.6, 0.0176) (0.7, 0.0105) (0.8, 5.816e-3) (0.9, 2.477e-3)};
\addplot coordinates { (0.01, 5.447) (0.1, 0.346) (0.2, 0.0967) (0.3, 0.0422) (0.4, 0.0222) (0.5, 0.0127) (0.6, 7.589e-3) (0.7, 4.474e-3) (0.8, 2.442e-3) (0.9, 1.032e-3)};
\addplot coordinates { (0.00321, 18.0) (0.1, 0.162) (0.2, 0.0419) (0.3, 0.0177) (0.4, 9.117e-3) (0.5, 5.162e-3) (0.6, 3.042e-3) (0.7, 1.779e-3) (0.8, 9.654e-4) (0.9, 4.060e-4)};
\legend {$\xi = 10^{-2}$, $\xi = 10^{-3}$, $\xi = 10^{-4}$, $\xi = 10^{-5}$}
\end{axis}
\end{tikzpicture}
\end{subfigure}
\begin{subfigure}[b]{0.47\textwidth}
\begin{tikzpicture}
\begin{axis}[ xlabel={$x$}, ylabel={$\kappa^{(1)}_{I}(x)$}, ymode = log, legend pos= north east ]
\addplot coordinates { (0.1, 1.312) (0.156, 1.841) (0.2, 1.551) (0.3, 0.797) (0.4, 0.431) (0.5, 0.254) (0.6, 0.162) (0.7, 0.110) (0.8, 0.0787) (0.9, 0.0589) (1, 0.0459)};
\addplot coordinates { (0.0316, 4.438) (0.0473, 22.088) (0.1, 8.621) (0.2, 1.722) (0.3, 0.605) (0.4, 0.277) (0.5, 0.148) (0.6, 0.0884) (0.7, 0.057) (0.8, 0.0391) (0.9, 0.0284) (1, 0.0216)};
\addplot coordinates { (0.01, 13.923) (0.0146, 232.678) (0.1, 6.021) (0.2, 0.946) (0.3, 0.301) (0.4, 0.130) (0.5, 0.0674) (0.6, 0.0391) (0.7, 0.0246) (0.8, 0.0166) (0.9, 0.0119) (1, 8.955e-3)};
\addplot coordinates { (0.00321, 45.084) (0.00477, 2409.301) (0.1, 3.037) (0.2, 0.429) (0.3, 0.130) (0.4, 0.0551) (0.5, 0.0279) (0.6, 0.0159) (0.7, 9.942e-3) (0.8, 6.643e-3) (0.9, 4.705e-3) (1, 3.508e-3)};
\legend {$\xi = 10^{-2}$, $\xi = 10^{-3}$, $\xi = 10^{-4}$, $\xi = 10^{-5}$}
\end{axis}
\end{tikzpicture}
\end{subfigure}
\caption{Example 2. Graphs of functions $\kappa^{(k)}_{I}(x),$ $k=0,1$ that correspond to different values of parameter $\xi;$ $h=10^{-4}.$} \label{Example_2_graph_kappa_I}
\end{figure}
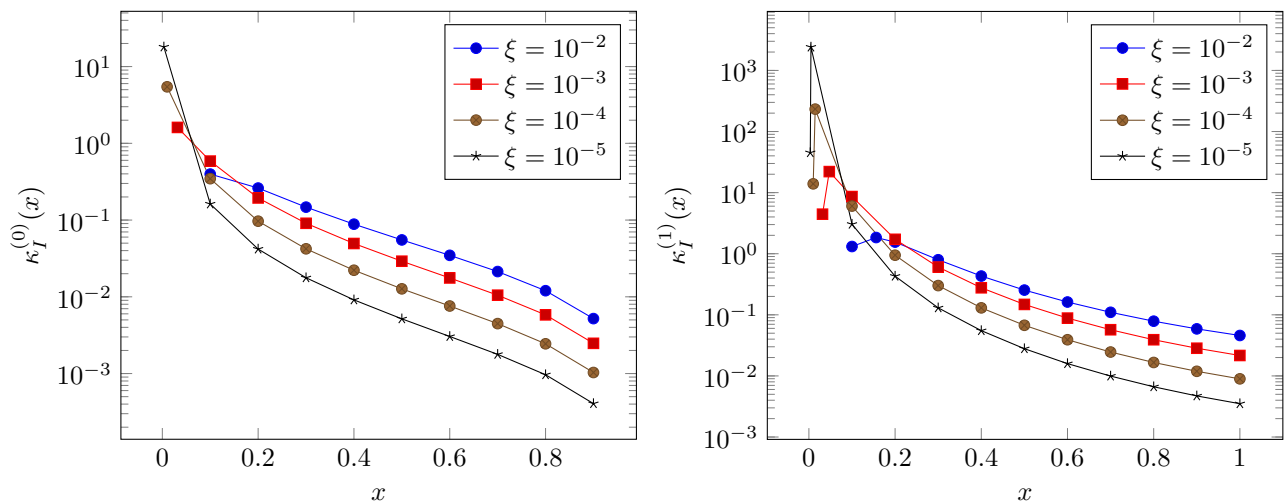

The results of numerical experiments are presented in Tab. \ref{table_bvpT21}, featuring pretty much the same set of parameters as in the previous numerical example, although some of the parameters have a bit different meanings, as clarified below. As one can easily conclude, $N_{S}$ and $N_{I}$ denote number of discretization knots that belong to the "straight" $[c, 1]$ and "inverse" $[\tilde{u}(c), 1]$ intervals respectively. For the case of  \verb"bvpT21" we redefine $\kappa^{(k)}_{S}$ and $\kappa^{(k)}_{S}(x)$ \eqref{Example_1_def_of_kappa_func} as
$$\kappa^{(k)}_{S} = \|\kappa^{(k)}_{S}(x)\|_{[c, 1]} = h^{-2} \|u^{(k)}(x) - \tilde{u}^{(k)}(x)\|_{[c, 1]},\; k=0,1,$$
whereas the definition of $\kappa^{(k)}_{I}$ and $\kappa^{(k)}_{I}(x)$ technically remains the same (see \eqref{Example_1_def_of_kappa_func}).

Graphs on Fig. \ref{Example_2_graph_kappa_S}, \ref{Example_2_graph_kappa_I} give general understanding of how the approximation characteristics of the SI-method vary throughout the intervals $[c, 1]$ and $[\tilde{u}(c), 1]$ respectively. The overall picture conforms to what we have seen with the Troesch's problem: the accuracy of the SI-method's approximation increases as we move away from the "critical" point.

\subsection{Example 3.} By means of the third example we want to push the applicability boundaries of the SI-method even further, applying it to the following problem:

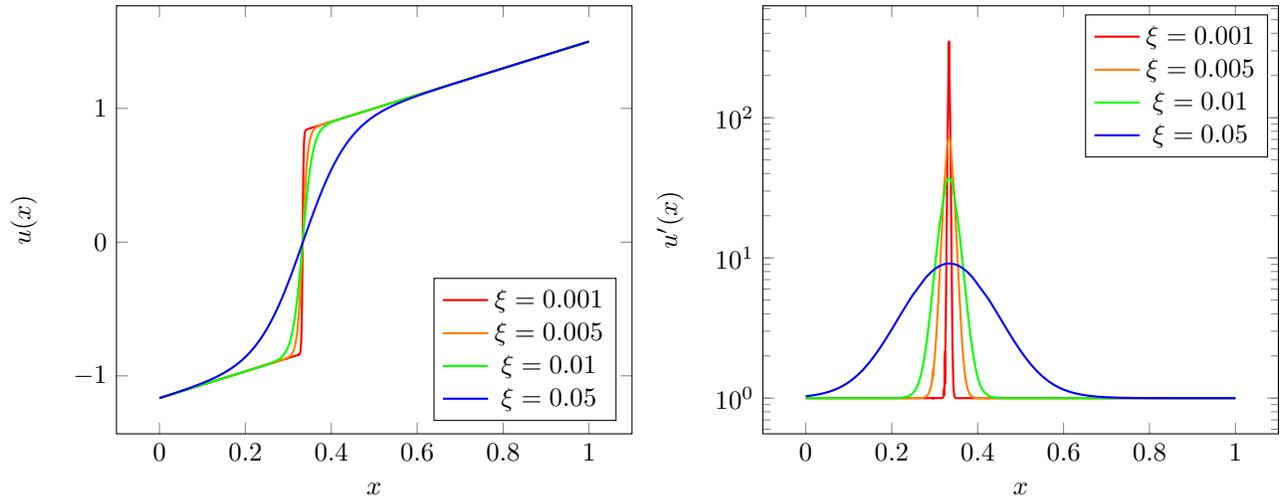
\begin{figure}[h!]
\centering
\begin{subfigure}[b]{0.47\textwidth}
\begin{tikzpicture}
        \begin{axis}[xlabel={$x$}, ylabel={$u(x)$}, legend pos=south east]
            \addplot[mark = none, red, thick] table {bvp_t_30_lambda_0.001_function.txt};
            \addplot[mark = none, orange, thick] table {bvp_t_30_lambda_0.005_function.txt};
            \addplot[mark = none, green, thick] table {bvp_t_30_lambda_0.01_function.txt};
            \addplot[mark = none, blue, thick] table {bvp_t_30_lambda_0.05_function.txt};
            \legend {$\xi = 0.001$, $\xi = 0.005$, $\xi = 0.01$, $\xi = 0.05$}
        \end{axis}
    \end{tikzpicture}
\end{subfigure}
\begin{subfigure}[b]{0.47\textwidth}
\begin{tikzpicture}
        \begin{axis}[ymode = log, xlabel={$x$}, ylabel={$u^{\prime}(x)$}]
            \addplot[mark = none, red, thick] table {bvp_t_30_lambda_0.001_derivative.txt};
            \addplot[mark = none, orange, thick] table {bvp_t_30_lambda_0.005_derivative.txt};
            \addplot[mark = none, green, thick] table {bvp_t_30_lambda_0.01_derivative.txt};
            \addplot[mark = none, blue, thick] table {bvp_t_30_lambda_0.05_derivative.txt};
            \legend {$\xi = 0.001$, $\xi = 0.005$, $\xi = 0.01$, $\xi = 0.05$}
        \end{axis}
    \end{tikzpicture}
\end{subfigure}
\caption{Example 3. Graphs of solutions $u(x)$ of problem \eqref{bvp_t_30} (left) and their derivatives $u^{\prime}(x)$ (right) for different values of $\xi.$} \label{Example_3_graph_solution_derivative}
\end{figure}

\begin{equation}\label{bvp_t_30}
  \xi u^{\prime\prime}(x) = (1-u^{\prime}(x))u(x),\; x\in [0,1],
\end{equation}
$$u(0) = -7/6,\; u(1) = 3/2,$$

which is known in literature as \verb"bvpT30" \cite{Cash_Algo_927}. As one can conclude from Fig. \ref{Example_3_graph_solution_derivative}, the stiffness of the problem is determined by the interval of rapid variation of its solution, which lies somewhere inside $(0, 1)$ and shrinks as parameter $\xi > 0$ tends to $0,$ while the magnitude of the variation remains almost constant. The general idea of the SI-method suggests that in such a case we should introduce a pair of "critical" points $c_{1}, c_{2}\in (0,1),$ $c_{1} < c_{2}$ so that the interval of rapid variation of solution $u(x)$ is enclosed inside $[c_{1}, c_{2}]$ and the "hybrid" problem should be restated as
$$\mathfrak{u}^{\prime\prime}(x) = \mathcal{N}(\mathfrak{u}^{\prime}(x), \mathfrak{u}(x), x), \; x\in [0, c_{1}] \cup [c_{2}, 1],\; \mathfrak{u}(0) = -7/6,\; \mathfrak{u}(1) = 3/2,\; \mathfrak{u}(c_{1}) \neq \mathfrak{u}(c_{2}),$$
$$\mathfrak{x}^{\prime\prime}(x) = -\mathcal{N}(1/\mathfrak{x}^{\prime}(u), u, \mathfrak{x}(u))\left(\mathfrak{x}^{\prime}(u)\right)^{3}, \; u\in [\mathfrak{u}(c_{1}), \mathfrak{u}(c_{2})],$$
$$\mathfrak{x}(\mathfrak{u}(c_{i})) = c_{i},\; \mathfrak{x}^{\prime}(\mathfrak{u}(c_{i})) = \frac{1}{\mathfrak{u}^{\prime}(c_{i})},\; i=1,2, \; \mathcal{N}(u^{\prime}, u, x) = \xi^{-1}(1-u^{\prime})u.$$
With this in mind, and some obvious modifications of the implementation guide described in Section \ref{section_implementation_aspects}, the SI-method can be successfully applied to problem \eqref{bvp_t_30} and the corresponding results are presented in Tab. \ref{table_bvpT30} and \ref{table_bvpT30_multiprecision}.

Similarly to the two preceding examples, we are interested in estimating quantities $\kappa_{S}^{(k)},$ $\kappa_{I}^{(k)},$ $k=0,1,$ defined as
$$ \kappa_{S}^{(k)} = h^{-2} \|u^{(k)}(x) - \tilde{u}^{(k)}(x)\|_{[0, c_{1}] \cup [c_{2}, 1]},$$
$$ \kappa_{I}^{(k)} = h^{-2} \|x^{(k)}(u) - \tilde{x}^{(k)}(u)\|_{[\tilde{u}(c_{1}), \tilde{u}(c_{2})]}, \; k=0,1,$$
for different values of $\xi.$ We define (compare to \eqref{criteria_to_calculate_c})
$$c_{1} = \min\limits_{i}\{x_{i}\; : \; |\tilde{u}^{\prime}(x_{i})| \geq u^{\prime}_{crit.}\},$$
$$c_{2} = \max\limits_{i}\{x_{i}\; : \; |\tilde{u}^{\prime}(x_{i})| \geq u^{\prime}_{crit.}\}$$
and choose $u^{\prime}_{crit.} = 2$ (having $u^{\prime}_{crit.} = 1$, as in the previous examples, does not make any sense since, as it can be seen from Fig. \ref{Example_3_graph_solution_derivative}, $u^{\prime}(x) \geq 1,$ $\forall x \in [0,1]$).

To get a reference approximation of the exact solution $u(x)$ (together with its first derivative), we used {\bf bvpSolve} package which is available through {\bf R} environment, see \cite{Cash_bvpSolve_in_R_2014}, \cite{Cash_bvpSolve_test_problems_2010}. For that purpose, \verb"bvptwp" subroutine was applied to the problem with the following set of parameters (for all the trial values of parameter $\xi$) :
$$\mathtt{x} = \mathtt{seq}(0, 1, \mathtt{by} = 0.0001),\; \mathtt{atol} = 1e-10,\; \mathtt{order} = 2,\; \mathtt{nmax} = 1000000.$$ Execution times of the subroutine (for different values of parameter $\xi$) were measured and are presented in the columns $T^{\ast}$ of Tab. \ref{table_bvpT30} and \ref{table_bvpT30_multiprecision}, right next to the timings $T,$ representing execution times of our implementation of the SI-method being applied to problem \eqref{bvp_t_30} (for the same values of parameter $\xi$). Since \verb"bvptwp" subroutine's output is a collection of values of the unknown solution and its derivative at some number $N^{\ast}$ of points from $[0,1]$ (a mesh), it barely can be directly used to calculate quantities $\kappa_{S}^{(k)}$. The latter issue, however, can be solved by using, for example, a cubic splines interpolation. Given the set of parameters, mentioned above, the distance between two successive points of the mesh produced by \verb"bvptwp" should not be greater than $h = 0.0001.$ An interpolation by cubic splines ensures the approximation error be of order $h^{3},$ --- right enough to investigate discrepancies of order $h^{2}$ (which is the approximation order of the SI-method). The interpolation was implemented by means of subroutine \verb"ArrayInterpolation" from {\bf CurveFitting} package in {\bf Maple 2016} environment.

To get a reference approximation of function $x(u),$  which is inverse to the exact solution $u(x),$ we used \verb"dsolve" subroutine available in {\bf Maple 2016} environment. The subroutine was called for the boundary value problem
$$x^{\prime\prime}(u) = -\left(10 + (1/\xi - 10)l\right)\left(\left(x^{\prime}(u)\right)^{3} - \left(x^{\prime}(u)\right)^{2}\right)u,\; x(-7/6) = 0,\; x(3/2) = 1, $$
with the following set of parameters: $$\mathtt{numeric},\; \mathtt{output} = \mathtt{listprocedure},$$
$$ \mathtt{maxmesh} = 60000,\; \mathtt{continuation} = l, \mathtt{range} = -7/6 .. 3/2, \mathtt{abserr} = 10^{-16}.$$

\begin{table}
\begin{tabular}{|c|c|c|c|c|c|c|c|c|c|c|c|}
  \hline
  $\xi$ & $c_{1}$ & $c_{2}$ & $N^{\ast}$ & $N_{S}$ & $N_{I}$& $\kappa^{(0)}_{S}$ & $\kappa^{(1)}_{S}$ & $\kappa^{(0)}_{I}$ & $\kappa_{I}^{(1)}$ & $T^{\ast},$ sec & $T,$ sec\\
  \hline
   5e-2 & 0.16 & 0.509 & 1001 & 10380 & 21154 & 1.1 & 7.2 & 0.53 & 3.0 & 5.89 & 0.29 \\
   1e-2 & 0.276 & 0.390 & 1001 & 13433 & 19346 & 5.3 & 155.7 & 2.6 & 50.1 & 8.34 & 0.40 \\
   5e-3 & 0.3 & 0.366 & 1001 & 14181 & 18936 & 38.0 & 1692.5 & 46.4 & 405.8 & 11.2 & 2.0 \\
  \hline
\end{tabular}
\caption{Example 3. Results of numerical experiments for different values of $\xi$ and $h=10^{-4}.$ Double precision calculations. }\label{table_bvpT30}
\end{table}

\begin{table}
\begin{tabular}{|c|c|c|c|c|c|c|c|c|c|c|c|c|}
  \hline
  $\xi$ & $c_{1}$ & $c_{2}$ & $N^{\ast}$ & $N_{S}$ & $N_{I}$& $\kappa^{(0)}_{S}$ & $\kappa^{(1)}_{S}$ & $\kappa^{(0)}_{I}$ & $\kappa_{I}^{(1)}$ & $T^{\ast},$ sec & $T,$ sec & Digits\\
  \hline
   5e-3 & 0.3 & 0.366 & 10001 & 14110 & 18937 & 10.6 & 609.2 & 5.3 & 163.2 & 11.2 & 24.5 & 16\\
   4e-3 & 0.306 & 0.361 & 10028 & 14265 & 18841 & 13.2 & 954.2 & 6.6 & 250.6 & 12.6 & 28.1 & 16\\
   3e-3 & 0.312 & 0.355 & 10028 & 14418 & 18740 & 17.6 & 1674.4 & 8.8 & 439.3 & 14.8 & 49.8 & 22\\
   2e-3 & 0.318 & 0.349 & 20055 & 14594 & 18725 & 26.9 & 3836.7 & 13.4 & 996.8 & 27.6 & 114.3 & 35\\
   1e-3 & 0.325 & 0.342 & 641675 & 14771 & 18498 & 52.5 & 14846 & 25.9 & 3801.4 & 1794.5 & 738.8 & 60\\
  \hline
\end{tabular}
\caption{Example 3. Results of numerical experiments for different values of $\xi$ and $h=10^{-4}.$ Extended precision calculations. }\label{table_bvpT30_multiprecision}
\end{table}

The meaning of parameters $N_{S}$ and $N_{I}$ presented in Tab. \ref{table_bvpT30} and \ref{table_bvpT30_multiprecision} remains pretty much the same as in the previous numerical examples, namely, they denote the numbers of discretization knots on the "straight" $[0, c_{1}] \cup [c_{2}, 1]$ and "inverse" $[\tilde{u}(c_{1}), \tilde{u}(c_{2})]$ intervals respectively. $N^{\ast}$ denotes number of elements in the collection returned by \verb"bvptwp" subroutine (as a single element of the collection we consider a triplet consisting of an argument $x$ and the corresponding approximations of the unknown solution and its derivative at that argument).

As one can tell from the captions, the difference between the two tables mentioned above is in the implementation of the SI-method used to obtain the corresponding numerical results. Tab. \ref{table_bvpT30} deals with a "double precision" implementation (i.e. all the calculations are done using the built-in type \verb"double" of C++ programming language). The implementation referred to in Tab. \ref{table_bvpT30_multiprecision} is based on type \verb"number<cpp_dec_float<D>>" from \verb"boost::multiprecision" name space ({\bf Boost C++ Libraries} ver. 1.59.0), where the value of the integer template parameter \verb"D" was chosen according to the column "Digits" from the table. It is worth mentioning that the actual machine epsilon for the latter numerical type is typically much smaller than  $10^{-D}$ (although the latter is exactly the value returned by function \verb"std::numeric_limits::epsilon" for the type). For instance, the value of machine epsilon that corresponds to $D = 16$ is of order $10^{-33}$ (one can easily check this using the definition of the machine epsilon as the biggest positive value $\varepsilon$ such that $1+\varepsilon = 1$). All this means that to populate Tab. \ref{table_bvpT30_multiprecision} with the data we had to perform calculations in precision that is considerably higher than the "double" one. The reason for that is explained below.

The specifics of problem \eqref{bvp_t_30} is that the difference $1 - u^{\prime}(0)$ tends to $0$ (while always remaining positive) as $\xi > 0$ tends to $0.$ Our numerical experiments suggest that already for $\xi = 0.005$  the mentioned difference is less than the machine epsilon for the "double precision" arithmetics (that is, $1.11\times10^{-16}$). The latter makes it practically impossible (for low values of parameter $\xi$) to use the implementation approach described in Section \ref{section_implementation_aspects}, especially in the part about obtaining the initial guess via a single shooting procedure. A shooting by adjusting tangent of the unknown solution at point $x=0$ will always give us an approximation of function $u^{\dagger}(x) = x-7/6,$ which, apparently, satisfies the equation from \eqref{bvp_t_30} as well as the corresponding boundary condition at $x=0,$ while remaining quite far from the desired value $3/2$ at $x=1.$ We managed, however, to overcome this issue by using the output of the SI-method, applied to the problem with $\xi = 0.01,$ as an initial guess when solving the corresponding system of nonlinear equations (see Section \ref{section_implementation_aspects} and \cite{Makarov_Dragunov_2019} for more details about the implementation) for $\xi = 0.005$ (a, sort of, "chasing" approach). It is worth mentioning that, even with this "trick", we noticed that the corrections of the Newton's method applied to the mentioned nonlinear system stopped their convergence to $0$ at values of order $10^{-6},$ which might indicate about some numerical instabilities of our implementation being revealed by the problem in question. The instabilities have gone when the implementation was switched to do calculations via a numerical type of higher precision and the corresponding results are presented in Tab. \ref{table_bvpT30_multiprecision}. We leave the root cause analysis of the revealed instabilities for the further studies.

The data in the tables suggests that constants $\kappa_{S}^{(k)},$ $\kappa_{I}^{(k)},$ $k=0,1$ increase as $\xi$ decreases, which is in a good agreement with the results of the previous two numerical examples. As it can be seen from the execution time measurements (columns $T$ and $T^{\ast}$), when operating in double precision, our implementation of the SI-method performs better than that of \verb"bvptwp" subroutine. Switching to numerical types of higher precision essentially degrades performance of our implementation (which is expected) making it actually slower than \verb"bvptwp" for $\xi \in \left\{5\cdot10^{-3}, 4\cdot10^{-3}, 3\cdot10^{-3}, 2\cdot10^{-3}\right\}.$ The situation, however, changes dramatically for $\xi = 10^{-3},$ when the number of knots in the discretization mesh of \verb"bvptwp" suddenly increases in about 32 times: from $20055$ ($\xi = 2\cdot 10^{-3}$) to $641675,$ making \verb"bvptwp" subroutine about two and half times slower than the SI-method's implementation. At the same time, the number of knots in the mesh produced by the SI-method does not change much as $\xi$ decreases from $5\cdot10^{-3}$ to $10^{-3}.$ The latter, in our opinion, clearly indicates about the promising potential of the mesh generation strategy which naturally follows from the SI "ideology".

\section{Conclusions}\label{Section_conclusions}
In the present paper we have laid down a theoretical foundation for a new and very promising numerical method for solving stiff boundary value problems also known as the SI-method. We have established several fundamental facts revealing the method's properties as well as developed the corresponding proof methodologies which, in our opinion, can be easily enhanced to handle a much broader class of BVPs than that described as problem \eqref{Intro_Equation}, \eqref{Intro_boundary_conditions}. The results of numerical examples, presented in the paper, obviously support our optimism and uncover new theoretical and practical challenges to be addressed in the further studies.

\clearpage

\bibliography{../bibliography/references_stat}{}
\bibliographystyle{plain}
\end{document}